\documentclass[reqno,11pt]{amsart}

\usepackage{graphicx, amsmath, amssymb, amscd, amsthm, euscript, psfrag, amsfonts,bm}

\newtheorem{thm}[equation]{Theorem}\newtheorem{Thm}[equation]{Theorem}
\newtheorem{Ex}[equation]{Example}
\newtheorem{Q}[equation]{Problem}

\newtheorem{Con}[equation]{Conjecture}\newtheorem{Que}[equation]{Question}

\newtheorem{Def}[equation]{Definition}

\numberwithin{equation}{section}

\numberwithin{equation}{section}

\newcommand{\be}{begin{equation}}

\newcommand{\bH}{\mathbb H}

\newcommand{\q}{\mathbb{Q}}

\newcommand{\e}{\epsilon}
\newcommand{\z}{\mathbb{Z}}
\renewcommand{\q}{\mathbb{Q}}
\newcommand{\N}{\mathbb{N}}
\renewcommand{\c}{\mathbb{C}}
\newcommand{\br}{\mathbb{R}}

\newcommand{{\grinv}}{{\Cal G}^{-r}}

\newcommand{\A}{\mathcal{A}}
\newcommand{\ba}{\backslash}

\newcommand{\G}{\Gamma}

\newcommand{\Cal}{\mathcal}
\renewcommand{\P}{\mathcal P}

\newcommand{\SL}{\operatorname{SL}}

\newcommand{\bp}{\begin{pmatrix}}
\newcommand{\ep}{\end{pmatrix}}

\renewcommand{\bp}{{\rm bp}}

\renewcommand{\O}{\operatorname{O}}
\newcommand{\SO}{\operatorname{SO}}
\newcommand{\SU}{\operatorname{SU}}

\renewcommand{\H}{\mathcal{H}}
\newcommand{\T}{\operatorname{T}}

\newcommand{\B}{\mathcal B}

\newcommand{\PSL}{\op{PSL}}

\newcommand{\PS}{\rm{PS}}

\newcommand{\op}{\operatorname}

\renewcommand{\deg}{\text{DEP}}

\newcommand{\BR}{\operatorname{BR}}

\newcommand{\BMS}{\operatorname{BMS}}

\renewcommand{\be}{\begin{equation}}
\newcommand{\ee}{\end{equation}}

\newcommand{\Res}{\op{Res}}

\begin{document}

\title[Apollonian packings]{Apollonian circle packings: Dynamics and Number theory}

\author{Hee Oh}
\address{Mathematics department, Yale university, New Haven,  CT 06520
and Korea Institute for Advanced Study, Seoul, Korea}
\email{hee.oh@yale.edu}

\begin{abstract}
We give an overview of various counting problems for Apollonian circle packings, which turn out to be related to problems in dynamics
and number theory for thin groups. This survey article is an expanded version of my lecture notes prepared for the 13th Takagi lectures given at RIMS, Kyoto
 in the fall of 2013. 
\end{abstract}

\maketitle
\tableofcontents

\section{Counting problems for Apollonian circle packings}\begin{figure}  \begin{center}
  \includegraphics[height=4cm]{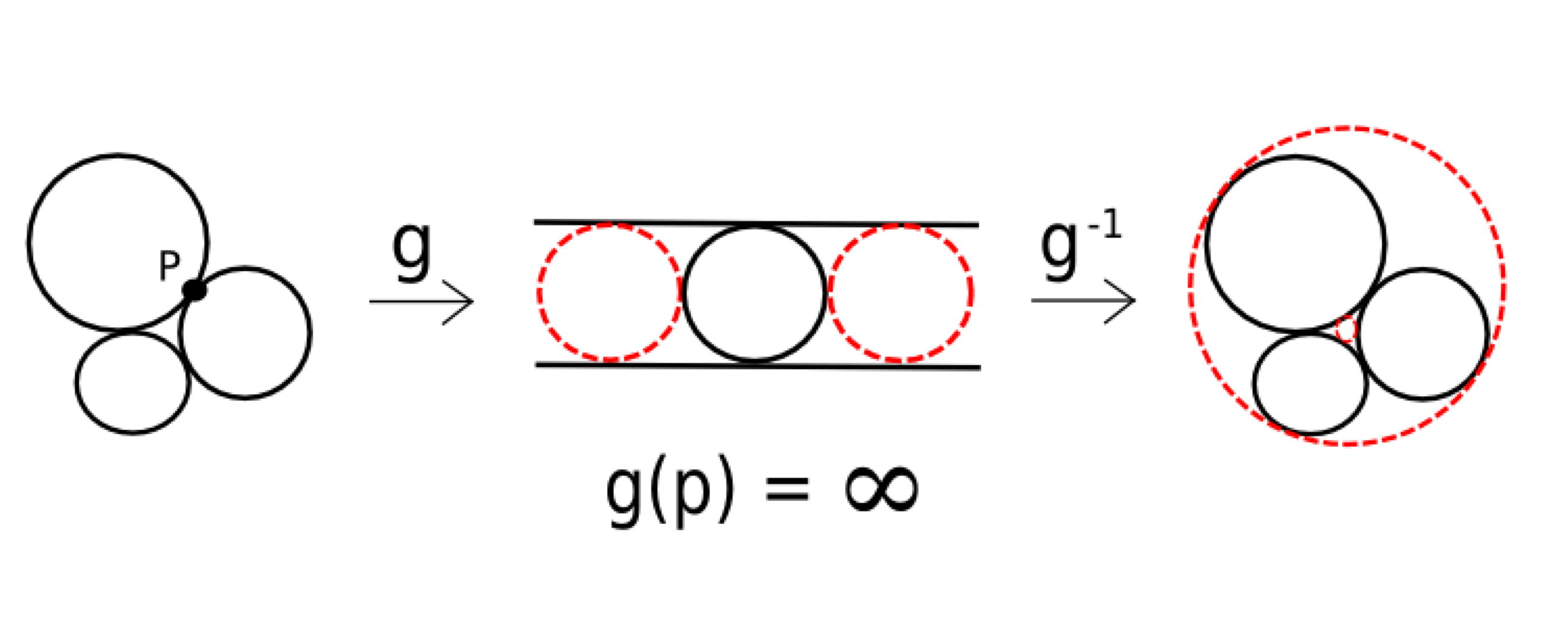}
  \end{center}
  \caption{Pictorial proof of the Apollonius theorem} \label{apr}
\end{figure}
An Apollonian circle packing is one of the most of beautiful circle packings whose construction can be described in a very simple
manner based on an old theorem of
 Apollonius of Perga:
\begin{Thm}[Apollonius of Perga, 262-190 BC]\label{apt}
Given $3$ mutually tangent circles in the plane, there exist exactly two circles tangent to all three.
\end{Thm}

\begin{proof}
We give a modern proof,
using  the linear fractional transformations of $\PSL_2(\c)$ on the extended complex plane $\hat \c=\c\cup\{\infty\}$, known as  M\"obius transformations:
$$\begin{pmatrix} a&b\\ c &d\end{pmatrix}(z)=\frac{az+b}{cz+d},$$
where $a,b,c,d\in \c$ with $ad-bc=1$ and $z\in \c\cup\{\infty\}$. 
As is well known, a M\"obius transformation maps
circles in $\hat \c$ to circles in $\hat \c$, preserving angles between them. (In the whole article, a line in $\c$ is treated as a circle in $\hat \c$).
In particular, it maps tangent circles to tangent circles. 

For given three mutually tangent circles $C_1, C_2, C_3$ in the plane, denote by $p$ the tangent point between $C_1$ and $C_2$, and 
let $g\in \PSL_2(\c)$ be an element which maps $p$ to $\infty$. Then $g$ maps $C_1$ and $C_2$ to two circles tangent at $\infty$, that is, two parallel lines, and
$g(C_3)$ is a circle tangent to these parallel lines. In the configuration of $g(C_1)$, $g(C_2)$, $g(C_3)$ (see Fig. \ref{apr}),
it is clear that there are precisely two circles, say, $D$ and $D'$ tangent to all three $g(C_i)$, $1\le i \le 3$.
Using $g^{-1}$, which is again a M\"obius transformation, it follows that $g^{-1}(D)$ and $g^{-1}(D')$ are precisely those two circles tangent to
$C_1$,$C_2$,$C_3$.
\end{proof}

\begin{figure}
\begin{center}
   \includegraphics[width=4in]{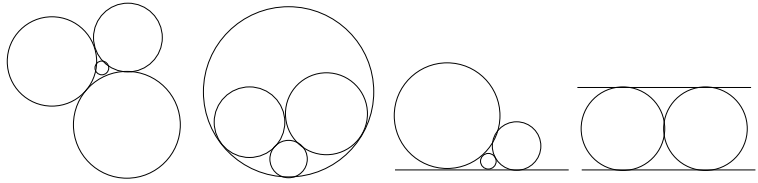}
 \end{center}
 \caption{Possible configurations of four mutually tangent circles}  \label{Fourcircles}
  \end{figure}

In order to construct an Apollonian circle packing, we begin with four mutually tangent circles in the plane (see  Figure \ref{Fourcircles} for
possible configurations) and keep adding newer circles tangent to three of the previous circles provided by Theorem \ref{apt}. Continuing this process
indefinitely, we arrive at an infinite circle packing,
called an {\it Apollonian circle packing}.

Figure \ref{stage} shows the first few generations of this process, where each circle is labeled with
its curvature (= the reciprocal of its radius) with the normalization that
 the greatest circle has radius one.

\begin{figure}
  \begin{center}
   \includegraphics[width=1in]{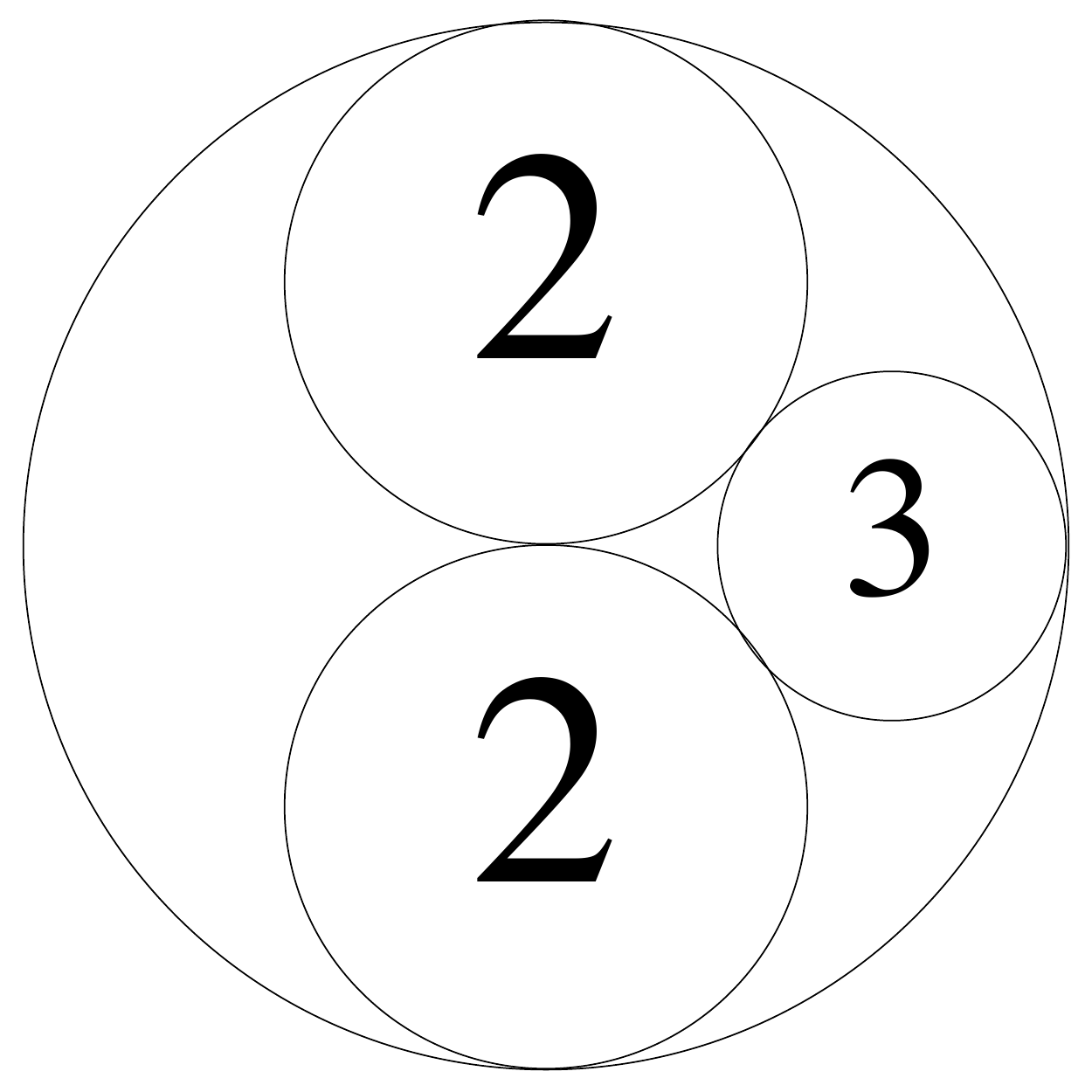}
    \includegraphics[width=1in]{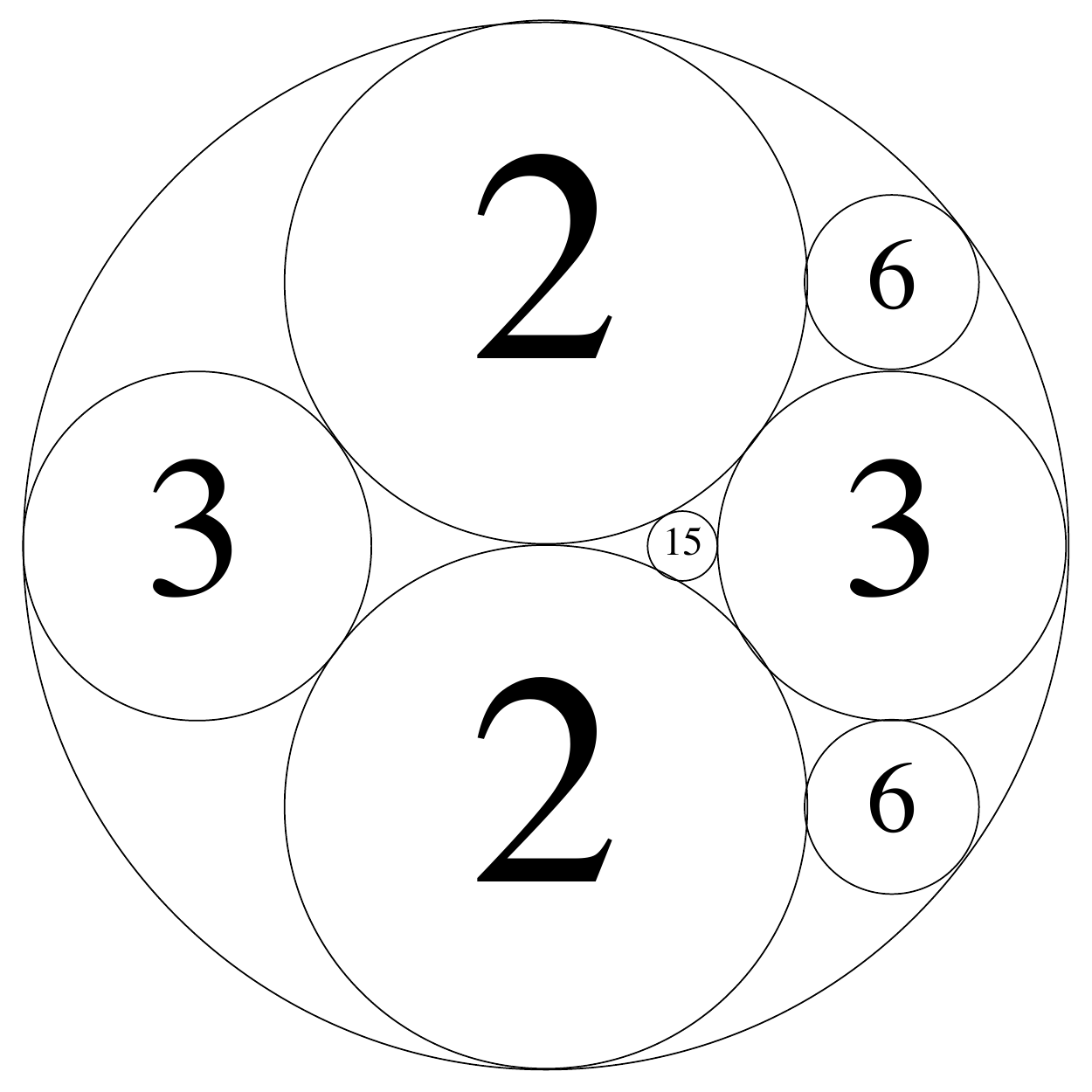}
 \includegraphics[width=1in]{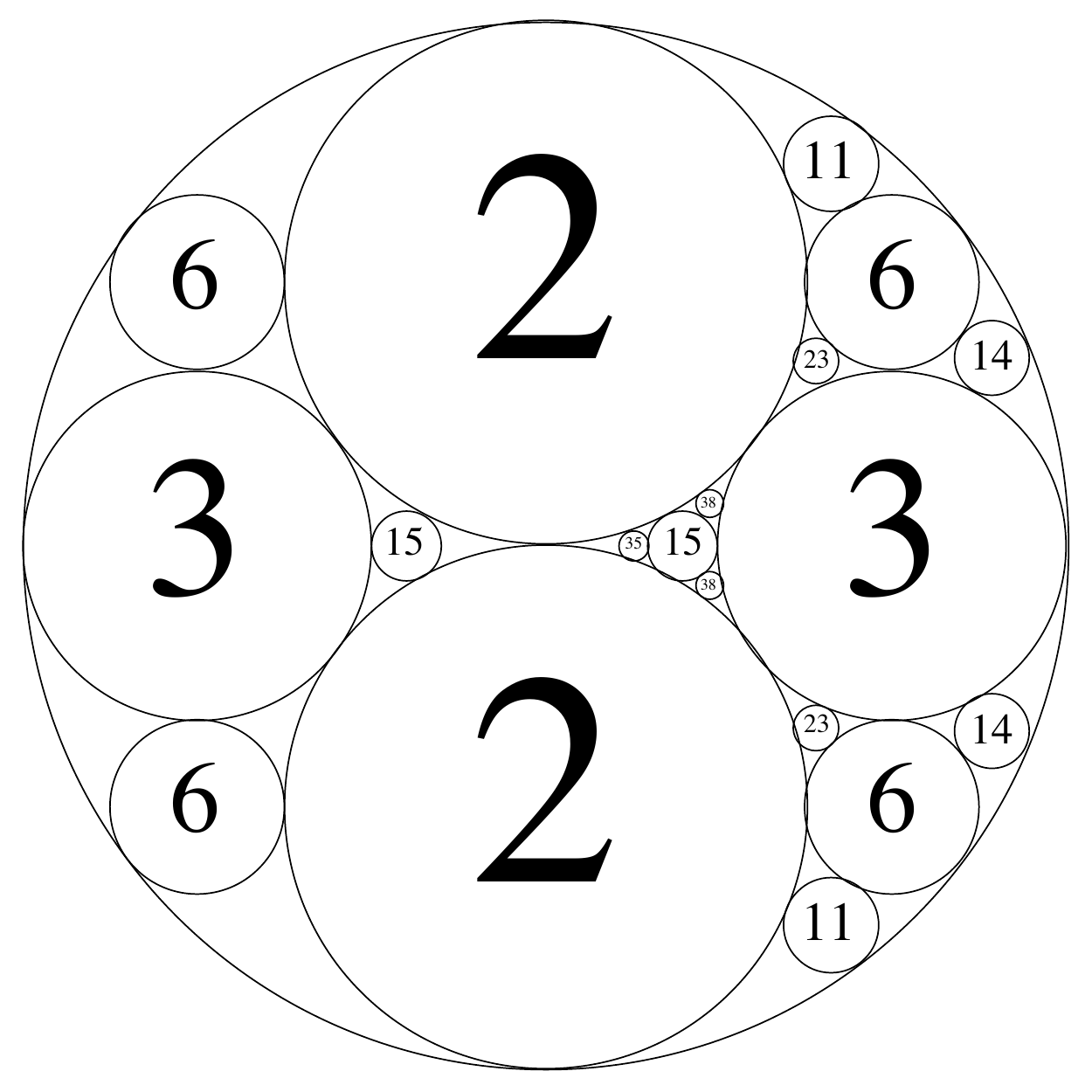}
 \includegraphics [width=1in]{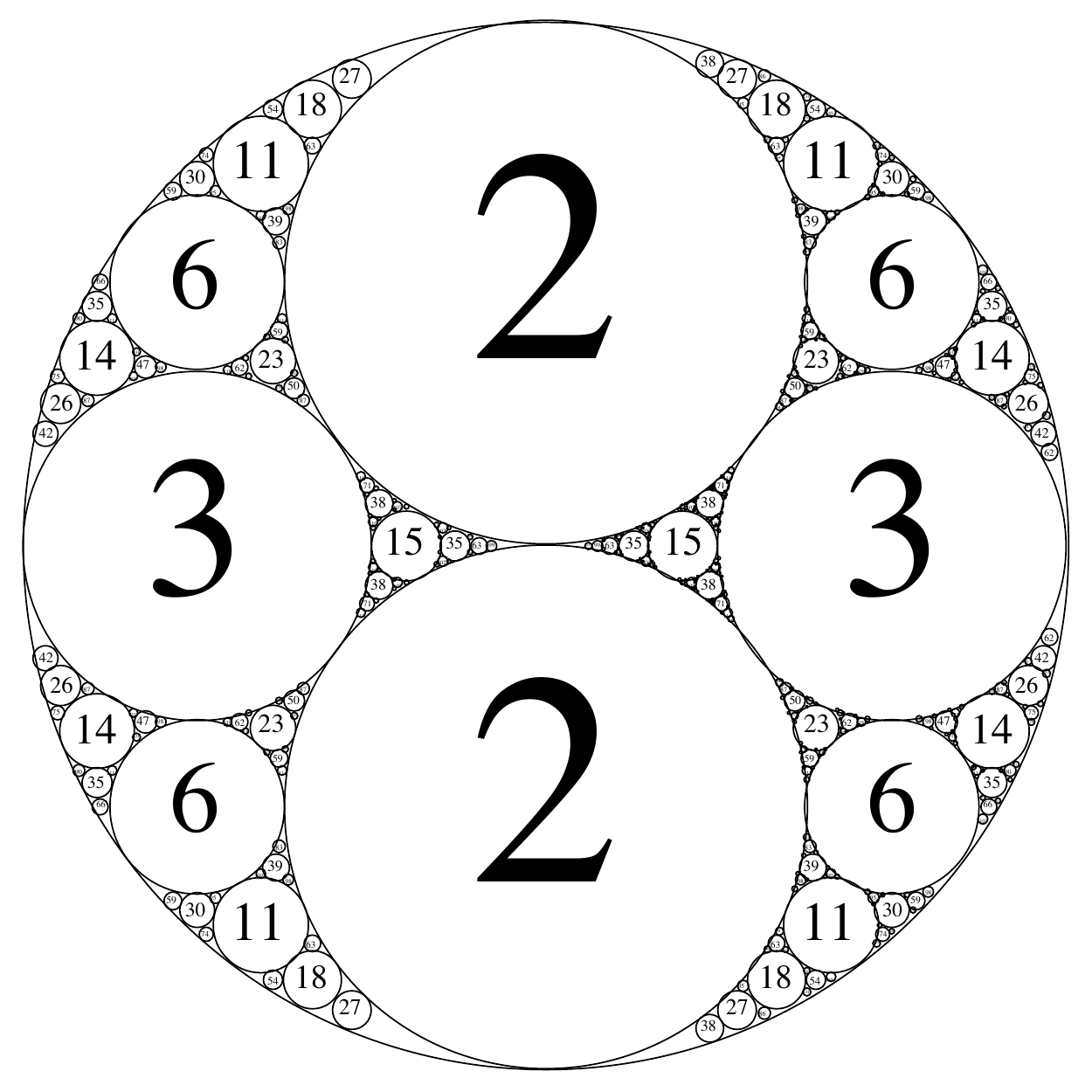}
\caption{First few generations}\label{stage} \end{center}
\end{figure}

 \begin{figure}\begin{center}
\includegraphics [width=2in]{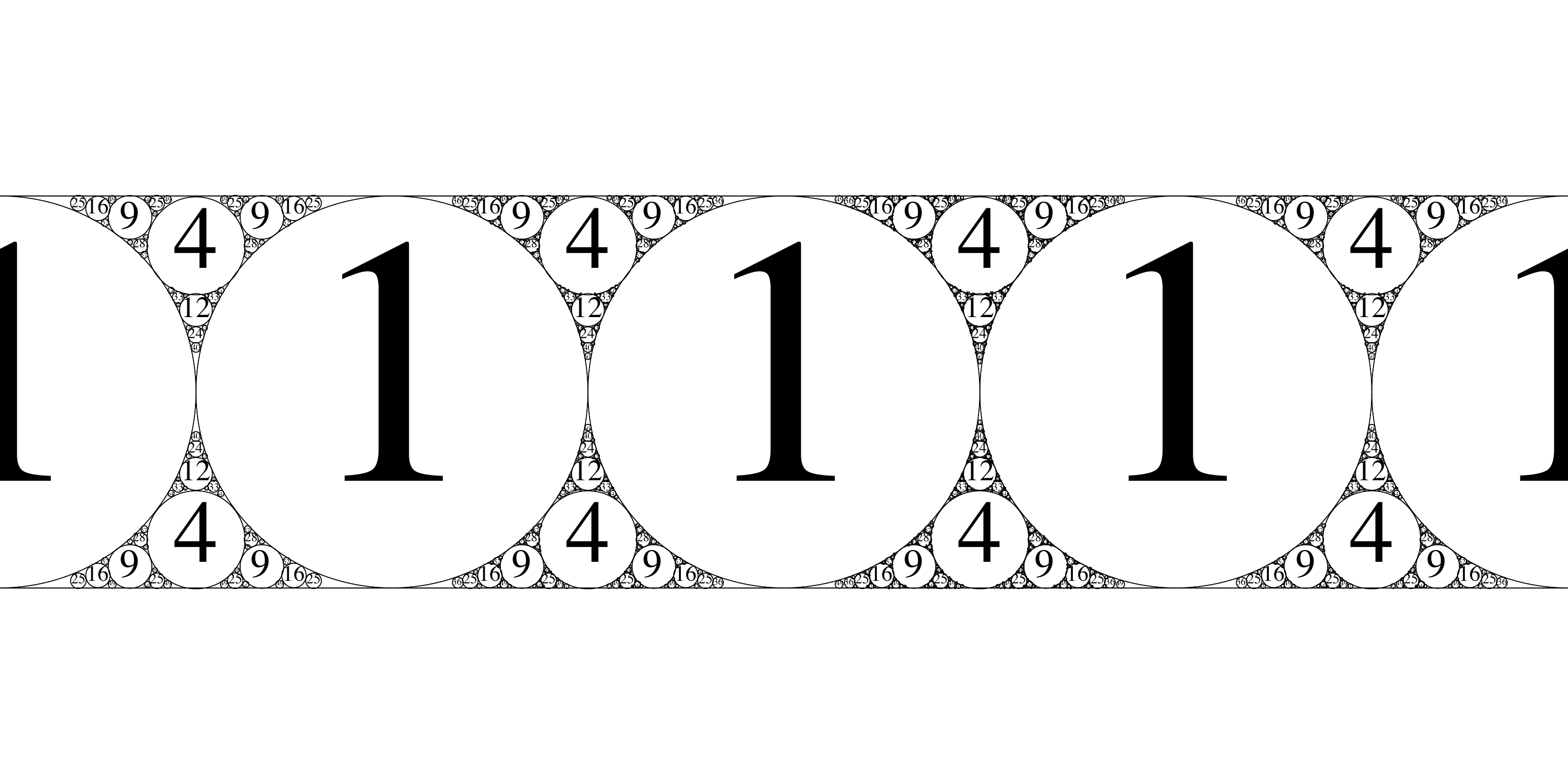}\caption{Unbounded Apollonian circle packing} \label{ea}
 \end{center}
\end{figure}

If we had started with a configuration containing two parallel lines, we would have arrived at an unbounded Apollonian 
circle packing as in Figure \ref{ea}.
There are also other unbounded Apollonian packings containing
either only one line or no line at all; but it will be hard to draw them in a paper with finite size, as circles will get enormously large only after a few generations.

For a bounded Apollonian packing $\mathcal P$,
there are only finitely many circles of radius bigger than a given number. Hence the following counting function is well-defined for any $T>0$:
$$N_{\mathcal P}(T):=\#\{C\in \mathcal P:\op{curv}(C)\le  T\} .$$

\begin{Que}
\begin{itemize}
\item Is there an asymptotic formula of $N_{\mathcal P}(T)$ as $T\to \infty$?
\item If so, can we compute?
\end{itemize}\end{Que}

The study of this question involves notions related to metric properties of the underlying fractal set called
a {\it residual set}:
$$\op{Res}({\P}):=\overline{\cup_{C\in \mathcal P} C} , $$
i.e., the residual set of $\P$ is the closure in $\c$ of the union of all circles in $\P$.

%Equivalently, the residual set of ${\P}$ is the fractal set which is left in the plane after removing
%all the open disks enclosed by circles in ${\P}$.

The Hausdorff dimension
 of the residual set of ${\P}$ is called
the  {\it residual dimension of ${\P}$}, which we denote
by $ \alpha $. The notion of the Hausdorff dimension was first given by Hausdorff in 1918. 
To explain its definition, we first recall the notion of the Hausdorff measure (cf. \cite{Mat}):
\begin{Def} 
Let $s\ge 0$ and $F$ be any subset of $ \mathbb R^n$. The $s$-dimensional Hausdorff measure  of $F$ is defined by
$$\mathcal H^s(F):=\lim_{\epsilon \to 0} \left(\inf \{\sum d(B_i)^s: F\subset \cup_i B_i, d(B_i)<\e\}\right)$$
where $d(B_i)$ is the diameter of $B_i$. \end{Def}

For $s=n$,  it is the usual Lebesgue measure of $\br^n$, up to a constant multiple.
It can be shown that as $s$ increases, the $s$-dimensional Hausdorff measure of $F$ will
be $\infty$ up to a certain value and then jumps down to $0$.
 The Hausdorff dimension of $F$ is 
this critical value of $s$:
$$\op{dim}_{\H}(F)=\sup\{s: \mathcal H^s(F)=\infty\}=
\inf \{s: \mathcal H^s(F)=0\}.$$

In fractal geometry, there are other notions of dimensions which
often have different values. But for the residual set of an Apollonian circle packing,
the Hausdorff dimension, the  packing dimension and the box dimension
 are all equal to each other \cite{StU}.

We observe
\begin{itemize}
\item $1\le \alpha\le 2$.
\item $\alpha$ is independent of $\mathcal P$: any two Apollonian packings are equivalent to each other by
a M\"obius transformation which maps three tangent points of one packing to three tangent points of the other packing.
\item The precise value of $\alpha$ is unknown, but approximately, $\alpha =1.30568(8)$ due to McMullen \cite{Mc}.
\end{itemize}

In particular, $\op{Res}({\P})$ is much bigger than a countable union of circles (as $\alpha>1$), but
not too big in the sense that its Lebesgue area is zero (as $\alpha <2$).

 The first counting result for Apollonian packings is due to Boyd in 1982 \cite{Boy}:
 \begin{Thm} [Boyd] $$\lim_{T\to \infty} \frac{\log N_{\mathcal P}(T)}{\log T} =\alpha .$$
\end{Thm}

Boyd asked in \cite{Boy} whether $N_{\mathcal P}(T)\sim c \cdot T^{\alpha}$ as $T\to \infty$,
and wrote that his numerical experiments suggest this may be false and
 perhaps
$$N_{\mathcal P}(T)\sim c \cdot T^\alpha (\log T)^\beta$$ might be more appropriate.

However it turns out that there is no extra logarithmic term:
\begin{Thm} [Kontorovich-O.  \cite{KO}]\label{kot}
For a bounded Apollonian packing $\mathcal P$, there exists a constant $c_{{\P}}>0$ such that
$$N_{\mathcal P}(T)\sim c_{\mathcal P} \cdot T^\alpha \quad\text{ as $T\to \infty$}.  $$
 \end{Thm}

\begin{Thm}[Lee-O. \cite{LOA}] \label{LOa}
There exists $\eta>0$ such that
for any bounded Apollonian packing $\mathcal P$, 
$$N_{\mathcal P}(T)= c_{\mathcal P} \cdot T^\alpha +O(T^{\alpha-\eta}) .$$ \end{Thm}

Vinogradov \cite{V} has also independently obtained Theorem \ref{LOa} with a weaker error term.

For  an unbounded Apollonian packing ${\P}$, we have $N_{{\P}}(T)=\infty$ in general; however we can
modify our counting question so that we count only those circles contained in a fixed curvilinear
triangle $\mathcal R$ whose sides are given by three mutually tangent circles.

\begin{figure}
 \begin{center}
    \includegraphics[width=1.5in]{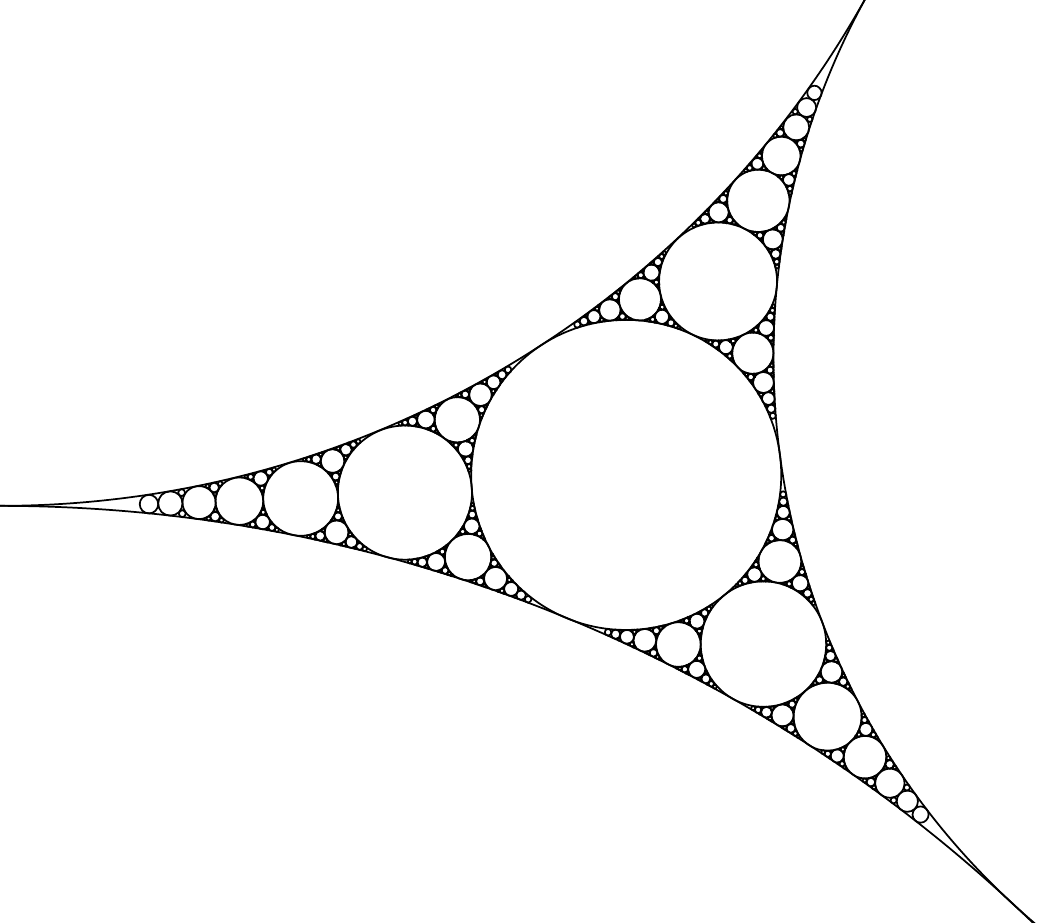}
 \end{center}
\end{figure}

Setting
$$N_{\mathcal R}(T):=\#\{C\in \mathcal R:\op{curv}(C)\le  T\} <\infty ,$$
we have shown:
\begin{Thm}[O.-Shah  \cite{OS1}]\label{os1}
For a curvilinear triangle $\mathcal R$ of
any Apollonian packing $\mathcal P$, there exists a constant $c_{\mathcal R}>0$ such that
$$N_{\mathcal R}(T)\sim c_{\mathcal R}  \cdot T^\alpha \quad\text{ as $T\to \infty$}.$$
\end{Thm}

Going even further, we may
ask if we can describe the  asymptotic distribution of circles in $\mathcal P$ of curvature at most $T$ as $T\to \infty$.
To formulate this question precisely, for any bounded region $E\subset \c$,
we set $$N_\P(T, E) :=\#\{C\in {\P}: C\cap E\ne\emptyset,\,\text{curv}(C) \le T\}.$$

Then  the question on the asymptotic distribution of circles in $\P$
 amounts to searching for a locally finite Borel measure $\omega_{\P}$ on the plane ${\mathbb C}$ satisfying that $$ \lim_{T\to \infty} \frac{N_T({\P}, E)}{T^\alpha} = \omega_{\P}(E) $$
 for any bounded Borel subset $E\subset {\mathbb C}$ with negligible boundary.

Noting  that all the circles in $\P$ lie on
the residual set of $\P$,  any Borel measure describing the asymptotic
distribution of 
circles of $\P$ must be supported on $\op{Res}(\P)$. 

\begin{Thm} [O.-Shah \cite{OS1}]  \label{dist}
For any bounded Borel $E\subset {\mathbb C}$ with smooth boundary,
$$ {N_\P(T, E)}\sim  c_A\cdot  {\mathcal H}^\alpha_\P(E ) \cdot {T^\alpha}\quad\text{ as $T\to \infty$}$$
 where ${\mathcal H}^\alpha_\P $ denotes  the $\alpha$-dimensional Hausdorff measure of the set $\Res(\P)$ and 
$0<c_A<\infty$ is  a constant independent of $\P$. 
  \end{Thm}

In general, $\text{dim}_\H(F)=s$ does not mean that the $s$-dimensional Hausdorff measure $\H^{s}(F)$ is non-trivial
(it could be $0$ or $\infty$). But on the residual set $\Res(\P)$ of an Apollonian packing,  $\H^{\alpha}_\P$ is known to be locally finite and its
support is precisely $\Res(\P)$ by Sullivan \cite{Sullivan1984}; hence ${\mathcal H}^\alpha_{\P} (E)<\infty$ for $E$ bounded and
  $0<{\mathcal H}^\alpha_{\P} (E)$ if $E^\circ\cap \Res(\P)\ne \emptyset$.

Though the Hausdorff dimension and
the packing dimension are equal to each other for $\Res (\P)$, the packing measure is locally infinite (\cite{Sullivan1984}, \cite{MU})
which indicates that
 the metric properties of $\Res (\P)$ are subtle.

Theorem \ref{dist} says that
 circles in an Apollonian packing $\P$ are  uniformly distributed with respect to
the $\alpha$-dimensional Hausdorff measure on $\Res(\P)$: for any bounded Borel subsets $E_1, E_2 \subset \c$ 
with smooth boundaries and with $E_2^\circ \cap \Res(\P)\ne \emptyset$,

$$\lim_{T\to \infty} \frac{N_\P(T, E_1)}{N_\P(T, E_2)} = \frac{\H^\alpha_\P (E_1)}{\H^\alpha_{\P}(E_2)} .$$

\noindent{\bf {Apollonian constant}.}
Observe that the constant $c_A$ in Theorem \ref{dist} is given by
 $$c_A=\lim_{T\to\infty}\frac{ N_\P(T, E)}{T^\alpha \cdot \H^\alpha_\P(E)} $$
for {\it any} Apollonian circle packing $\P$ and any $E$ with $E^\circ \cap \Res(\P)\ne \emptyset$.
In particular,  $$c_A=\lim_{T\to\infty}\frac{ N_\P(T)}{T^\alpha \cdot \H^\alpha_\P(\Res(\P))} $$
for any bounded Apollonian circle packing $\P$.
\begin{Def}\rm  We propose to call $0<c_A<\infty $ the {\it Apollonian constant}.
\end{Def}

\begin{Q}\rm What is $c_A$? (even approximation?)
\end{Q}

Whereas all other terms in the asymptotic formula of Theorem \ref{os1}
can be described using the metric notions of Euclidean plane,
our exact formula of $c_A$ involves certain singular measures
of an infinite-volume hyperbolic $3$ manifold, indicating the intricacy of the precise counting problem.

\bigskip
\begin{figure}
 \begin{center}
    \includegraphics[width=2in]{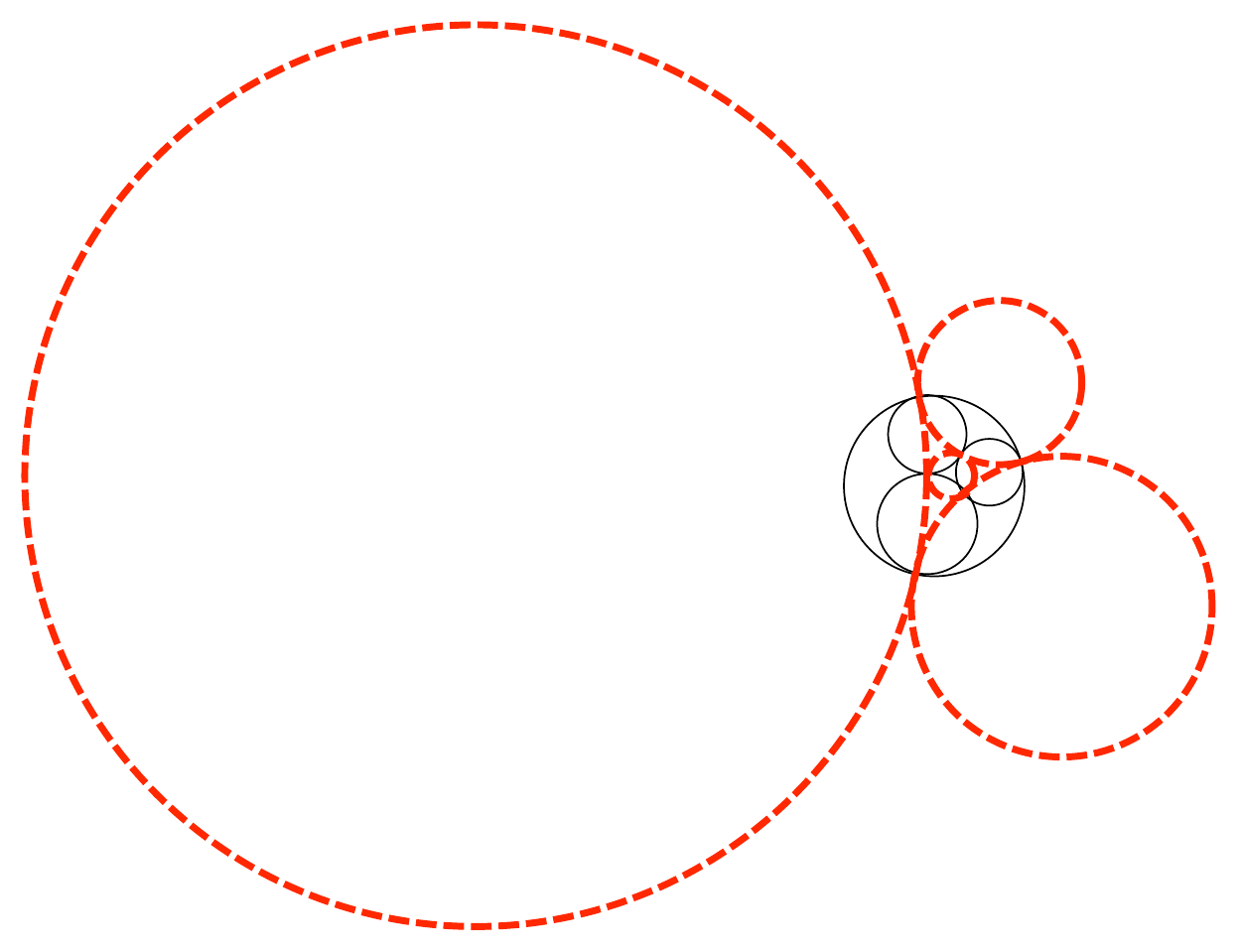}
 \end{center}\caption{Dual circles} \label{dcDual}
\end{figure}
%\section{Hidden symmetries and Dynamics on hyperbolic $3$ manifolds of infinite volume}
\section{Hidden symmetries and Orbital counting problem}\label{whitehead}
\noindent{\bf Hidden symmetries.}
The key to our approach of counting circles in an Apollonian packing lies in
 the fact that \begin{center}{\it An Apollonian circle packing has lots of
hidden symmetries}.\end{center}
Explaining these hidden symmetries will lead us to explain the relevance of the packing with
 a Kleinian group, called the (geometric) Apollonian group.

Fix
 4 mutually tangent circles $C_1, C_2, C_3, C_4$ in  $\P$ 
and consider their dual circles $\hat C_1, \cdots, \hat C_4$, that is,
 $\hat C_i$ is the unique circle passing through the three tangent points among $C_j$'s for $j\ne i$.
 In Figure \ref{dcDual}, 
 the solid circles represent $C_i$'s and the dotted circles are their dual circles.
Observe that inverting with respect to a dual circle preserves the three circles
that it meets perpendicularly and interchanges the two circles
which are tangent to those three circles.

\begin{Def} \rm The inversion 
 with respect to a circle of radius $r$ centered at $a$ maps $x$ to $a+\frac{r^2}{|x-a|^2} (x-a)$. 
 The group $\text{M\"ob}(\hat{\mathbb C})$ of M\"obis transformations in $\hat \c$  is generated by inversions with respect to all circles in $\hat \c$.
 \end{Def}

The geometric Apollonian group $\mathcal A:=\mathcal A_{\P}$ associated to $\P$ is
 generated by the four inversions with respect to the dual circles:
$$\mathcal A =\langle \tau_1, \tau_2, \tau_3, \tau_4 \rangle <\text{M\"ob}(\hat{\mathbb C})$$
where $\tau_i$ denotes the inversion with respect to $\hat C_i$.
Note that $\PSL_2(\c)$ is a subgroup of $\text{M\"ob}(\hat{\mathbb C})$ of index two; we will
write $\text{M\"ob}(\hat{\mathbb C})=\PSL_2(\mathbb C)^{\pm}$.
The Apollonian group $\A$ is a  Kleinian group (= a discrete subgroup of $\PSL_2(\mathbb C)^{\pm}$)
and satisfies 

\begin{itemize}
\item  $\P=\cup_{i=1}^4 \mathcal A (C_i)$, that is,
inverting the initial four circles in $\P$ with respect to their dual circles generates
the whole packing $\P$;
\item  $\op{Res}(\P)=\Lambda(\A)$ where $\Lambda(\A)$ denotes the limit set of $\A$, which is
the set of all accumulation points of an orbit $\mathcal A(z)$
for $z\in \hat{\mathbb C}$.
\end{itemize}

\begin{figure}
 \begin{center}
    \includegraphics[width=2in]{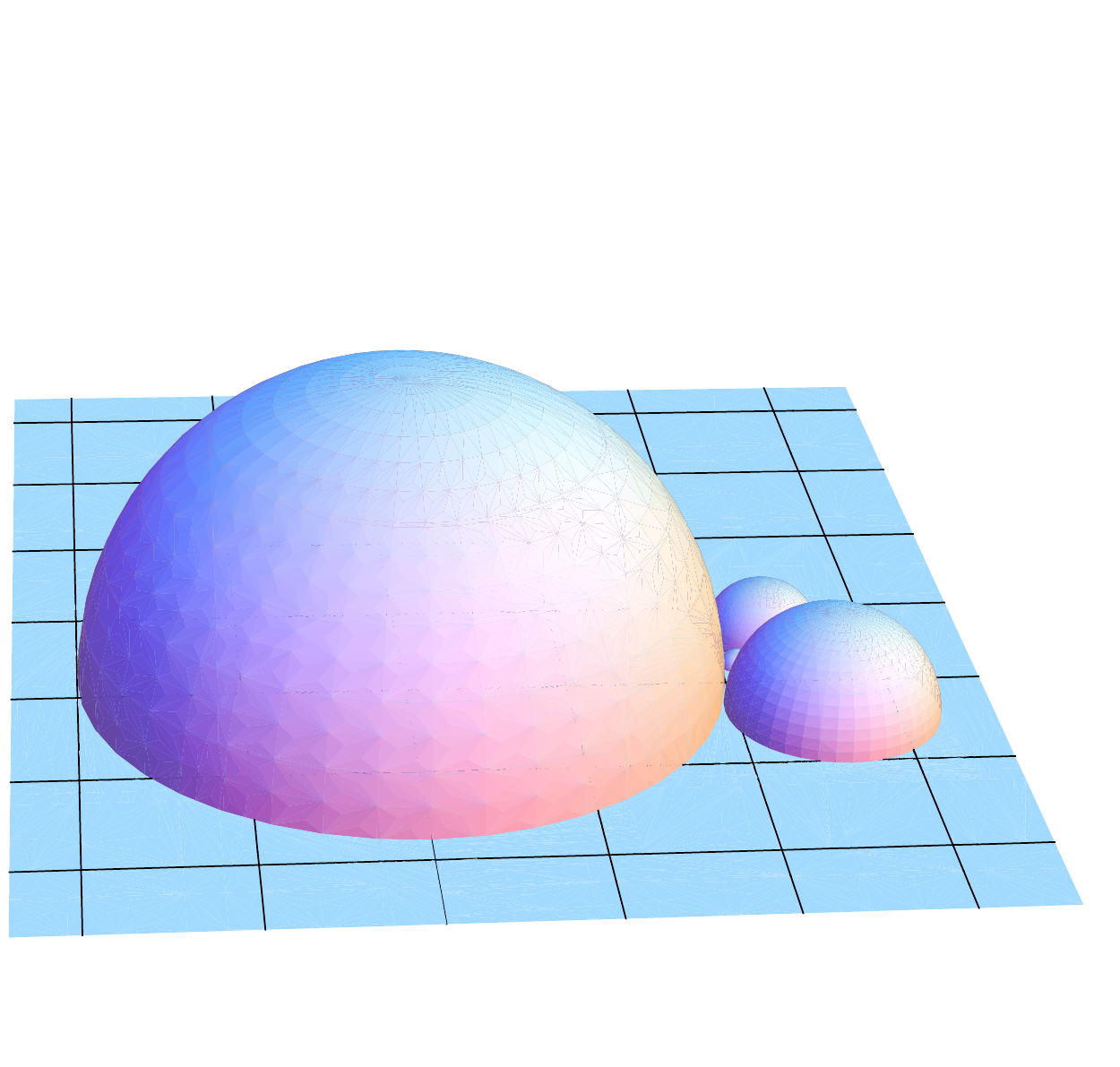}
 \end{center}
\end{figure}

In order to explain how the hyperbolic geometry comes into the picture, it is most convenient to use
the upper-half space model for hyperbolic $3$ space $\bH^3$:
$\bH^3=\{(x_1,x_2,y): y>0\}$. The hyperbolic metric is given by $ ds=\frac{\sqrt{dx_1^2+dx_2^2+dy^2}}{y}$
and the geometric boundary $\partial_\infty(\bH^3)$ is naturally identified with $\hat {\mathbb C}$.
Totally geodesic subspaces in $\bH^3$ are vertical lines, vertical circles, vertical planes, and vertical hemispheres.

%\begin{figure}
 %\begin{center}
   % \includegraphics[height=7cm]{im3}
 %\end{center}
%\end{figure}

The Poincare extension theorem gives an identification $\text{M\"ob}(\hat{\mathbb C})$ with the isometry group $\op{Isom}(\bH^3)$.
Since $\text{M\"ob}(\hat{\mathbb C})$ is generated by inversions with respect to circles in $\hat{\mathbb C}$, the Poincare extension theorem is determined by
the correspondence which assigns to an inversion with respect to a circle $C$ in $\hat{\mathbb C}$ the inversion with respect to the vertical
hemisphere in $\bH^3$ above $C$. An inversion with respect to a vertical hemisphere preserves the upper half space, 
as well as the hyperbolic metric, and hence
gives rise to an isometry of $\bH^3$.

%\begin{figure}
 %\begin{center}
   % \includegraphics[height=7cm]{pic4}
% \end{center}
%\end{figure}

The Apollonian group $\mathcal A=\mathcal A_\P$, now considered as a discrete subgroup of
$\op{Isom}(\bH^3)$, has a fundamental domain in $\bH^3$, given by
 the exterior of the hemispheres above the dual circles to $\P$.
In particular,
$\mathcal A\backslash \bH^3$ is an infinite volume hyperbolic $3$-manifold
and has a fundamental domain with finitely many sides; such a manifold is called a geometrically finite manifold. 

\begin{figure}
 \begin{center}
    \includegraphics[height=5cm]{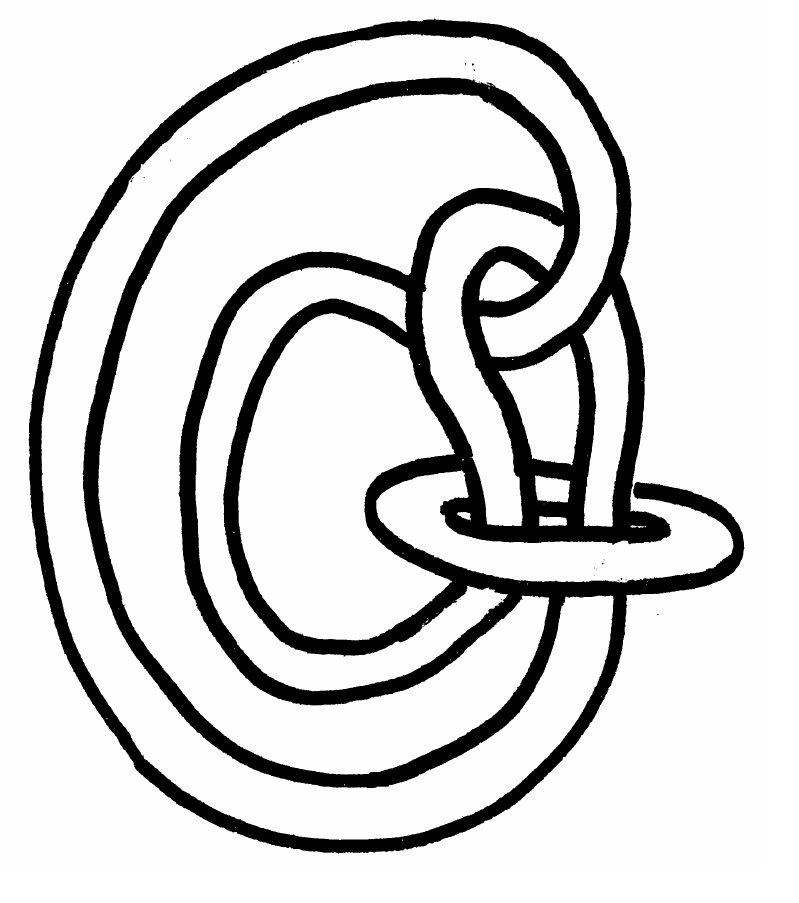}
 \end{center}
 \caption{Whitehead Link} \label{white}\end{figure}
 
 \medskip
 
\noindent{\bf Connection with the Whitehead link.}  
 The Apollonian manifold $\A\ba \bH^3$ can also be constructed from the Whitehead link complement. To explain the connection,
 consider the group, say, $\A^*$ generated by $8$ inversions with respect to
 four mutually tangent circles as well as their four dual circles. 
 Then the group $\A^*$ has a regular ideal hyperbolic octahedron as a fundamental domain in $\mathbb H^3$, and
 is commensurable to the Picard group $\PSL_2 (\z[i])$, up to a conjugation, which is a lattice in $\PSL_2(\c)$.
  The quotient orbifold  $\A^*\ba \bH^3$ is commensurable to the Whitehead link complement $S^3-W$ (see Figure \ref{white}).
  In this finite volume $3$-manifold $S^3-W$, we have a triply punctured sphere (corresponding
  a disk in $S^3$ spanning one component of $W$ and pierced twice by the other component),
  which is totally geodesic and whose fundamental group
  is conjugate to the congruence subgroup $\G(2)$ of $\PSL_2(\z)$ of level $2$.
  If we cut the manifold $S^3-W$ open along this totally geodesic surface $\G(2)\ba \bH^2$, 
  we get a finite volume hyperbolic manifold with
  totally geodesic boundary, whose fundamental group is the Apollonian group $\A$.
 We thank Curt McMullen for bringing this beautiful relation with the Whitehead link to our attention.

\medskip

\noindent{\bf Orbital counting problem in $\PSL_2(\br)\ba \PSL_2(\c)$.}
Observe that the number of circles in an Apollonian packing $\P$ of curvature at most $T$ is same as the number of
the vertical hemispheres  above circles in $\P$
of Euclidean height at least $T^{-1}$. Moreover
for a fixed bounded region $E$ in $\c$,
$N_\P(T, E)$ is same as the number of 
the vertical hemispheres above circles in $\P$
which intersects the cylindrical region 
\begin{equation}\label{et} E_T:=\{(x_1, x_2, y)\in \bH^3: x_1+ix_2\in E, \;\;T^{-1}\le  y\le r_0 \}\end{equation}
where $r_0>0$ is the radius of the largest circle in $\P$ intersecting $E$.

Since the vertical plane over the real line in $\c$ is preserved by $\PSL_2(\br)$, and $\PSL_2(\c)$ acts transitively
on the space of all vertical hemispheres (including planes),
the space of vertical hemispheres in $\bH^3$ can be identified with the homogeneous space $\PSL_2(\br)\ba \PSL_2(\c)$. Since $\P$ consists of finitely many $\A$-orbits of circles
in $\c$, which corresponds to finitely many $\mathcal A$-orbits of points in $\PSL_2(\br)\ba \PSL_2(\c)$,
understanding the asymptotic formula of $N_\P(T, E)$ is a special case of the following more general counting problem:
letting $G=\PSL_2(\c)$ and $H=\PSL_2(\br)$,
for a given sequence of growing compact subsets $\B_T$ in $H\ba G$ and a discrete $\mathcal A$-orbit $v_0 \A$ in $H\ba G$,
\begin{center}\it what is the asymptotic formula of the number $\# \B_T\cap v_0\A$?\end{center}
If $\A$ were of finite co-volume in $\PSL_2(\c)$, this type of question is well-understood due to the works of 
 Duke-Rudnick-Sarnak \cite{DRS}  and Eskin-McMullen \cite{EM}.
In the next section, we describe analogies/differences of this counting problem for discrete subgroups of infinite covolume.

\section{Counting, Mixing, and the Bowen-Margulis-Sullivan measure}\label{bmss}
\noindent{\bf Euclidean lattice point counting} We begin with a simple example of the lattice point counting problem in Euclidean space.
Let $G=\br^3$, $\Gamma =\z^3$ and let $B_T:=\{x\in \br^3: \|x\|\le T\}$ be the Euclidean ball of radius $T$ centered at the origin. 
In showing the well-known fact
$$\# \z^3 \cap B_T\sim \frac{4\pi}{3} T^3,$$ one way is to
count the $\G$-translates of a fundamental domain, say $\mathcal F:=[-\tfrac 12 , \tfrac 12)\times [-\tfrac 12, \tfrac 12)\times
[-\tfrac 12, \tfrac 12)$ contained in $B_T$, since each translate $\gamma  +\mathcal F$ contains precisely one
point, that is, $\gamma$, from $\Gamma$.
We have
\begin{multline}\label{eucc}
\frac{\op{Vol}(B_{T-1})}{\op{Vol}(\mathcal F)}\le \# \{\gamma + \mathcal F\subset B_{T-1}: \gamma\in \z^3\} 
 \\ \le \# \z^3 \cap B_T \le \# \{\gamma + \mathcal F\subset B_{T+1}: \gamma\in \z^3\} \le \frac{\op{Vol}(B_{T+1})}{\op{Vol}(\mathcal F)} ,$$
\end{multline} 

Since $\frac{\op{Vol}(B_{T\pm 1})}{\op{Vol}(\mathcal F)} =\frac{4\pi}{3} (T\pm 1)^3$,
we obtain that $$\# \z^3 \cap B_T =\frac{4\pi}{3} T^3+O(T^2).$$

This easily generalizes to the following: for any discrete subgroup $\Gamma$ in $\br^3$ and a sequence $B_T$ of compact subsets in $\br^3$,
we have $$\# \Gamma \cap B_T=\frac{\op{Vol}(B_{T})}{\op{Vol}(\Gamma\ba \br^3)}  +O(\op{Vol}(B_T)^{1-\eta})$$ 
provided
\begin{itemize}
 \item $\op{Vol}(\Gamma\ba \br^3)<\infty$;
 \item $\op{Vol} (\text{unit neighborhood of $\partial(B_T)$}) =O(\op{Vol}(B_T)^{1-\eta})$ for some $\eta>0$.
\end{itemize}

We have used here that the volume in $\br^3$ is computed with respect to the Lebesgue measure
which is clearly left $\Gamma$-invariant so that it makes sense to write $\op{Vol}(\Gamma\ba \br^3)$,
and that
the ratio $\tfrac{\op{Vol} (\text{unit neighborhood of $\partial(B_T)$})}{\op{Vol}(B_T)}$ tends to $0$ as $T\to \infty$.

\medskip

\noindent{\bf Hyperbolic lattice point counting}  We now consider the hyperbolic lattice counting problem for $\bH^3$.
Let $G=\PSL_2(\c)$ and $\Gamma$ be a torsion-free, co-compact, discrete subgroup of $G$. 
The group $G$ possess a Haar measure $\mu_G$ which is both left and right invariant under $G$, in particular, it is 
left-invariant under $\Gamma$. By abuse of notation, we use the same notation $\mu_G$ for the induced measure
on $\G\ba G$. Fix $o=(0,0,1)$ so that $g\mapsto g(o)$ induces an isomorphism of $\bH^3$ with $G/\operatorname{PSU}(2)$ and hence
$\mu_G$ also induces a left $G$-invariant measure on $\bH^3$, which will again be denoted by $\mu_G$.
Consider the hyperbolic ball $B_T=\{x\in \bH^3: d(o, x) \le T\}$ where $d$ is the hyperbolic distance in $\bH^3$.
Then, for a fixed fundamental domain $\mathcal F$ for $\G$ in $\bH^3$ which contains $o$ in its interior,
we have inequalities similar to \eqref{eucc}: 
\begin{multline}\label{naive}  \frac{\op{Vol}(B_{T-d})}{\op{Vol}(\Gamma\ba G )} \le
\# \{\gamma(\mathcal F)\subset B_{T-d}: \gamma\in \Gamma \} 
\\ \le \# \G(o)\cap B_T \le  \# \{\gamma(\mathcal F)\subset B_{T+d}: \gamma\in \Gamma \}
\le \frac{\op{Vol}(B_{T+d}) }{\op{Vol}(\Gamma\ba G)}
\end{multline}
where $d$ is the diameter of $\mathcal F$ and the volumes $\op{Vol}(B_{T\pm d})$
and $\op{Vol}(\Gamma\ba G)$ are computed with respect to $\mu_G$ on $\bH^3$ and $\G\ba G$ respectively.

If we had $\op{Vol}(B_{T-d})\sim \op{Vol}(B_{T+d})$ as $T\to \infty$ as in the Euclidean case, we would be able to conclude from here that
$\# \G (o) \cap B_T \sim \frac{ \op{Vol}(B_{T})}{\op{Vol}(\Gamma\ba G)}$ from \eqref{naive}. However, 
one can compute that $\op{Vol}(B_{T})\sim c \cdot e^{2T}$ for some $c>0$ and hence the asymptotic formula $\op{Vol}(B_{T-d})\sim \op{Vol}(B_{T+d})$ 
 is not true. This
suggests that
the above inequality \eqref{naive} gives too crude estimation of the edge effect arising from the intersections of 
$\gamma (\mathcal F)$'s with $B_T$ near the boundary of $B_T$.
It turns out that the mixing phenomenon of the geodesic flow on the unit tangent bundle $\T^1(\Gamma\ba \bH^3)$ 
 with respect to $\mu_G$
precisely clears out the 
fuzziness of the edge effect.
The mixing of the geodesic flow follows from
the following mixing of the frame flow, or equivalently, the decay of matrix coefficients due to Howe and Moore \cite{HM}:
Let $a_t:=\begin{pmatrix}
                       e^{t/2} & 0\\ 0 &e^{-t/2}
                      \end{pmatrix}$.
\begin{Thm} [Howe-Moore]  Let $\G<G$ be a lattice. For any $\psi_1, \psi_2\in C_c(\Gamma \ba G)$,
                      $$\lim_{t\to \infty}
                      \int_{\G\ba G}  \psi_1(ga_t) \psi_2(g) d\mu_G(g) = \frac{1}{\mu_G(\Gamma \ba G)} 
                      \int_{\G\ba G} \psi_1 d\mu_G \cdot \int_{\G\ba G} \psi_2 d\mu_G. $$ 
 \end{Thm}

%\begin{figure}
 %\begin{center}
   % \includegraphics[height=7cm]{p11}
 %\end{center}
%\end{figure}
Indeed, using this mixing property of the Haar measure, there are now very well established counting result due to Duke-Rudnick-Sarnak \cite{DRS}, and Eskin-McMullen\cite{EM}):
we note that any symmetric subgroup of $G$ is locally isomorphic to $\SL_2(\br)$ or $\SU(2)$.
\begin{Thm} [Duke-Rudnick-Sarnak, Eskin-McMullen] \label{fc}  Let $H$ be  a symmetric subgroup of $G$ and $\G<G$ a lattice such that $\mu_H(\G\cap H\ba H)<\infty$, i.e., $H\cap \G$ is a lattice in $H$. 
 Then for any well-rounded sequence $B_T$ of compact subsets in $H\ba G$ and a discrete $\G$-orbit $[e]\G$,
 we have $$\# [e]\G\cap B_T \sim \frac{\mu_H((H\cap \G) \ba H)}{\mu_G(\G\ba G)} \cdot \op{Vol}(B_T) \quad\text{as $T\to \infty$. }$$
 Here the volume of $B_T$ is computed with respect to the invariant measure $\mu_{H\ba G}$ on $H\ba G$ which satisfies
 $\mu_G=\mu_H\otimes \mu_{H\ba G}$ locally.
\end{Thm} A sequence $\{B_T\subset H\ba G \}$ is called {\it well-rounded} with respect to a measure $\mu$ on $H\ba G$ if the boundaries of $B_T$ are $\mu$-negligible,
more precisely, if for all small $\e>0$,
the $\mu$-measure of the $\e$-neighborhood of the boundary of $B_T$ is $O (\e \cdot \mu (B_T))$ as $T\to \infty$.

The idea of using the mixing of the geodesic flow in the counting problem goes back to Margulis' 1970 thesis
(translated in \cite{Ma}).

We now consider the case when $\G<G=\PSL_2(\c)$ is not a lattice, that is, $\mu_G(\G\ba G)=\infty$.
It turns out that as long as we have a left $\G$-invariant measure, say, $\mu$ on $G$,
satisfying 
\begin{itemize}
 \item  $\mu(\G\ba G)<\infty$; 
 \item $\mu$ is the mixing measure for the frame flow on $\G\ba G$, \end{itemize}
then the above heuristics of comparing the counting function for $\# \G (o)\cap B_T$ to the volume $\mu (B_T)$ can be made into a proof.

For what kind of discrete groups $\G$, do we have a left-$\G$-nvariant measure on $G$ satisfying these two conditions?
Indeed when $\G$ is geometrically finite,
the Bowen-Margulis-Sullivan measure $m^{\BMS}$ on $\G\ba G$ satisfies these properties. Moreover when $\G$ is convex cocompact (that is, geometrically finite with no parabolic elements), the Bowen-Margulis-Sullivan measure
is supported on a compact subset of $\G\ba G$.
Therefore $\G$ acts co-compactly in the convex hull $CH(\Lambda(\G))$ of the limit set $\Lambda(\G)$; recall that $\Lambda(\G)$
is the set of all accumulation points of $\Gamma$-orbits on the boundary $\partial(\bH^3)$.
  Hence if we denote by $\mathcal F_0$ a compact fundamental  domain for $\G$ in $CH(\Lambda(\G))$,
the inequality \eqref{naive} continues to hold if we replace the fundamental domain $\mathcal F$ of $\G$ in $\bH^3$
by $\mathcal F_0$ 
and compute the volumes with respect to $m^{\BMS}$:

\begin{multline}\label{naive22}  \frac{\tilde m^{\BMS}(B_{T-d_0})}{m^{\BMS}(\Gamma\ba G )} \le
\# \{\gamma(\mathcal F_0)\subset B_{T-d_0}: \gamma\in \Gamma \} 
\\ \le \# \G(o)\cap B_T \le  \# \{\gamma(\mathcal F_0)\subset B_{T+d_0}: \gamma\in \Gamma \}
\le \frac{\tilde m^{\BMS}(B_{T+d_0}) }{m^{\BMS} (\Gamma\ba G)}
\end{multline}
where $\tilde m^{\BMS}$ is the projection to $\bH^3$ of the lift of $m^{\BMS}$ to $G$ and $d_0$ is the diameter of $\mathcal F_0$.
This suggests a heuristic expectation:
$$\# \G(o)\cap B_T \sim \frac{\tilde m^{\BMS}(B_{T}) }{m^{\BMS} (\Gamma\ba G)}$$
which turns out to be true.

We denote by $\delta$ the Hausdorff dimension of $\Lambda(\G)$ which is known to
be equal to the critical exponent of $\G$.
Patterson \cite{Pa} and Sullivan\cite{Su} constructed a unique geometric probability measure $\nu_o$ on $\partial (\bH^3)$ satisfying 
that for any $\gamma \in \G$, $\gamma_* \nu_o$ is absolutely continuous with respect to $\nu_o$
and for any Borel subset $E$,
$$\nu_o(\gamma(E))=\int_{E} \left( \frac{d(\gamma_* \nu_o)}{d\nu_o}\right)^\delta  d\nu_o.$$
This measure $\nu_o$ is called the Patterson-Sullivan measure viewed from $o\in \bH^3$.
Then the Bowen-Margulis-Sullivan measure $m^{\BMS}$ on $\T^1(\bH^3)$ is given by 
$$dm^{\BMS}(v)= f(v)\; d\nu_o(v^+) d\nu_o(v^-)dt$$
where $v^{\pm}\in \partial(\bH^3)$ are the forward and the backward endpoints of the geodesic determined by $v$ and $t=\beta_{v^-}(o, v)$ measures the signed distance
of the horopsheres based at $v^-$ passing through $o$ and $v$. The density function $f$  is given by
$f(v)=e^{\delta (\beta_{v^+}(o,v) +\beta_{v^-}(o, v))}$ so that $m^{\BMS}$ is left $\G$-invariant 
Clearly, the support of $m^{\BMS}$ is given by the set of $v$ with $v^{\pm}\subset \Lambda(\G)$. Noting that $\T^1(\bH^3)$ is isomorphic to $G/M$ where
$M=\{\text{diag}(e^{i \theta}, e^{-i\theta})\}$, we will extend $m^{\BMS}$ to an $M$-invariant measure on $G$. 
We use the same notation $m^{\BMS}$ to denote the measure induced on $\G\ba G$.

\begin{Thm} For $\G$ geometrically finite and Zariski dense,
\begin{description}
 \item[(1) Finiteness] $m^{\BMS}(\G\ba G)<\infty$
 \item[(2) Mixing]  For any $\psi_1, \psi_2\in C_c(\Gamma \ba G)$, as $t\to \infty$,
                      $$
                      \int_{\G\ba G} \psi_1(ga_t) \psi_2(g) dm^{\BMS} (g) \to \tfrac{1}{m^{\BMS} 
                      (\Gamma \ba G)} \int_{\G\ba G} \psi_1  dm^{\BMS} \; \int_{\G\ba G} \psi_2  dm^{\BMS} .$$ 
 \end{description}
\end{Thm}
The finiteness result (1) is due to Sullivan \cite{Su} and the mixing result (2) for frame flow
 is due to Flaminio-Spatzier \cite{FS} and Winter \cite{Wi} based on the work of Rudolph \cite{Ru} and Babillot \cite{Ba}.

In order to state an analogue of Theorem \ref{fc} for a general geometrically finite group, we need to impose a 
condition on $(H\cap \G)\ba H$ analogous to the finiteness of the volume
$\mu_H(\G\cap H\ba H)$. In \cite{OS}, we define the so called skinning measure $\mu_H^{\PS}$ on $(\Gamma\cap H)\ba H$, which is intuitively the slice measure
on $H$ of $m^{\BMS}$. We note that $\mu_H^{\PS}$ depends on $\G$, not only on $H\cap \G$. A finiteness criterion for $\mu_H^{\PS}$ is given in \cite{OS}.
The following is obtained in \cite{OS} non-effectively and \cite{MO2} effectively.
\begin{Thm}  [O.-Shah, Mohammadi-O.]  \label{ifc} Let $H$ be  a symmetric subgroup of $G$ and $\G<G$ a geometrically finite and Zariski dense subgroup. Suppose that the skinning measure of
 $H\cap \G \ba H$ is finite, i.e., $\mu_H^{\PS} (\G\cap H\ba H)<\infty$.
 Then there exists an explicit locally finite Borel measure $\mathcal M_{H\ba G}$ on $H\ba G$
  such that for any well-rounded sequence $B_T$ of compact subsets in $H\ba G$ with respect to $\mathcal M_{H\ba G}$ and a discrete $\G$-orbit $[e]\G$,
 we have $$\# [e]\G\cap B_T \sim \frac{\mu_H^{\PS}((H\cap \G) \ba H)}{m^{\BMS} (\G\ba G)} \cdot \mathcal M_{H\ba G} (B_T)  \quad\text{as $T\to \infty$}. $$
\end{Thm} A special case of this theorem implies Theorem \ref{dist}, modulo the computation of the measure $\mathcal M_{H\ba G} (E_T)$ where $E_T$ is given in \eqref{et}. We mention that in the case when the critical exponent $\delta$ of $\G$ is strictly bigger than $1$, both Theorem \ref{ifc} and
Theorem \ref{dist} can be effectivized by \cite{MO2}.

The reason that we have the $\alpha$-dimensional Hausdorff measure in the statement of Theorem \ref{dist}
 is because the slice measure of $m^{\BMS}$ on each horizontal plane is the Patterson-Sullivan measure multiplied with
 a correct density function needed for the $\Gamma$-invariance, which
 turns out to coincide with the $\delta$-dimensional Hausdorff measure on the limit set of $\Lambda(\G)$
 when all cusps of $\G$ are of rank at most $1$, which is the case for the Apollonian group.

Counting problems for $\G$-orbits in $H\ba G$ are technically much more involved when $H$ is non-compact than when $H$ is compact, and relies on understanding the asymptotic distribution of
$\G\ba \G Ha_t$ in $\G\ba G$ as $t\to \infty$. When $H=\PSL_2(\br)$, the translate
$\G\ba \G Ha_t$ corresponds to the orthogonal translate of a totally geodesic surface for time $t$, and we showed that,
after the correct scaling of $e^{(2-\delta)t}$, $\G\ba \G Ha_t$ becomes equidistributed in $\G\ba G$ with respect to the Burger-Roblin measure $m^{\BR}$, which is the unique non-trivial
 ergodic horospherical invariant measure on $\G \ba G$. We refer to \cite{OhICM}, \cite{Oh1}, \cite{OS1} for more details.

 %The reason that we have the $\alpha$-dimensional Hausdorff measure in the statement of \ref{dist}
 %is whose conditional measures
%on horospherical foliations are $\alpha$-dimensional  Hausdorff  measures, and this
% is why we have the $\alpha$-dimensional Hausdorff measure in our statements of the counting theorem \ref{dist}.
%Main ingredients of our proofs include
%\begin{itemize} \item The
% {Lax-Phillips} spectral theory
%for the Laplacian on $\mathcal A\backslash \bH^3$;
%\item 
%Ergodic properties of flows on
%$\op{T}^1(\mathcal A\backslash \bH^3)$ based
%on the {Patterson-Sullivan} theory and
%the work of { Burger-Roblin} on classification of measures invariant under horospherical foliations.
%\end{itemize}
\bigskip

\noindent{\bf More circle packings.} This viewpoint of approaching Apollonian circle packings via the study of Kleinian groups allows us to deal with
more general circle packings, provided they are  invariant under
a non-elementary  geometrically finite Kleinian group.

One way to construct such circle packings is as follows:
\begin{Ex}\label{exc} \rm
Let $X$ be a finite volume hyperbolic $3$-manifold with non-empty totally geodesic boundary.
Then\begin{itemize}
\item $\G:=\pi_1(X)$ is a geometrically finite Kleinian group;
\item By developing $X$ in the upper half space $\bH^3$,
the domain of discontinuity $\Omega(\G):=\hat {\mathbb C} -\Lambda(\G)$
consists of the disjoint union of open disks (corresponding to the boundary
components of the universal cover $\tilde X$).
\end{itemize}

Set $\P$ to be the union of circles which are boundaries
of the disks in $\Omega(\G)$. In this case,
$\op{Res}(\P)$ defined as the closure of all circles in $\P$ is equal to the limit set $\Lambda(\G) $.
\end{Ex} In section \ref{whitehead}, we explained how Apollonian circle packings can be described in this way.

\begin{figure}
 \begin{center}
    \includegraphics[height=7cm]{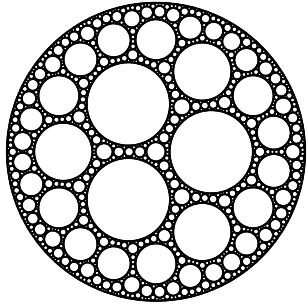}
 \end{center}
 \caption{Sierpinski curve} \label{sierp}\end{figure}

 Figure \ref{sierp}, due to McMullen, is also an example of a circle packing obtained in this way, here
 the symmetry group $\G$ is the fundamental group of a compact hyperbolic $3$-manifold with totally geodesic boundary being a compact surface of genus two.
  This limit set is called  a Sierpinski curve, being
  homeomorphic to the well-known Sierpinski Carpet.

Many more pictures of circle packings constructed in this way can be found in the book "Indra's pearls"
by  Mumford, Series and Wright  (Cambridge Univ. Press 2002).

For $\P$ constructed in Example \ref{exc},  we define as before $N_\P(T, E):=\#\{C\in \P: C\cap E\ne \emptyset, \text{curv}(C) \le T\}$
for any bounded Borel subset $E$ in $\c$.
\begin{Thm}[O.-Shah, \cite{OS}] 
There exist a constant $c_\Gamma>0 $ and a locally finite Borel measure $\omega_\P$ on $\Res(\P)$ such that
 for any bounded Borel subset $E\subset \mathbb C$ with $\omega_\P(\partial(E))=0$,
 $$N_\P(T, E)\sim c_\Gamma \cdot \omega_{\P} (E) \cdot T^{\delta} \quad\text{as $T\to \infty$} $$
where $\delta=\op{dim}_\H (\Res (\P))$. Moreover, 
 if $\G$ is convex cocompact or if the cusps of $\G$ have rank at most $1$, then $\omega_\P$ 
 coincides with the $\delta$-dimensional
Hausdorff measure on $\Res(\P)$.
\end{Thm}
We refer to \cite{OS} for  the statement for more general circle packings.
\begin{figure} \begin{center}
 \includegraphics [height=7cm]{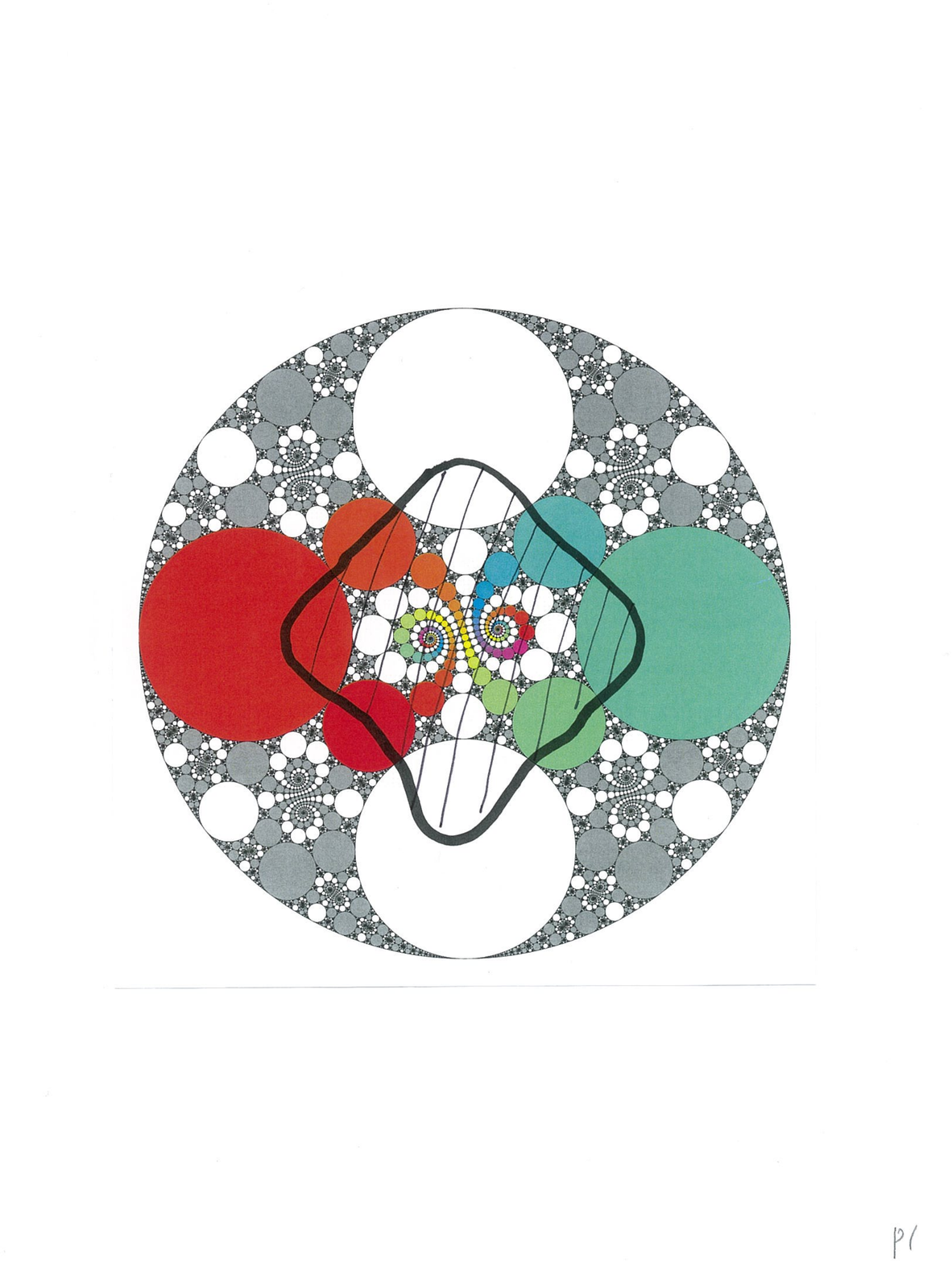}
 \end{center}
\end{figure}

\section{Integral Apollonian circle packings}
We call an Apollonian circle packing $\P$
 {\it integral} if every circle in $\P$ has integral curvature.
Does there exist  {\it any} integral $\P$?
The answer is positive thanks to the following beautiful theorem of Descartes: 
\begin{Thm}[Descartes 1643, \cite{Co}] A quadruple $(a,b,c,d)$ is the curvatures of four mutually tangent circles
if and only if it satisfies the
quadratic equation:
$$2(a^2+b^2+c^2+d^2)=(a+b+c+d)^2 .$$\end{Thm}

In the above theorem, we ask circles to be oriented so that their interiors are
disjoint with each other. For instance, according to this rule, the quadruple of curvatures of
 four largest four circles in Figure \ref{stage} is $(-1,2,2,3)$ or $(1,-2,-2,-3)$,
 for which we can easily check the validity of the Descartes theorem: $ 2((-1)^2+2^2+2^2+3^2)=36 =(-1+2+2+3)^2$

 In what follows, we will always assign  the negative curvature to the largest bounding circle
 in a bounded Apollonian packing, so that all other circles will then have positive curvatures.

Given three mutually tangent circles of curvatures $a,b,c$,
the curvatures, say, $d$ and $d'$, of the two circles tangent to all three
must satisfy $ 2(a^2+b^2+c^2+d^2)=(a+b+c+d)^2$  and $2(a^2+b^2+c^2+(d')^2)=(a+b+c+d')^2$ by the
Descartes theorem. By subtracting the first equation from the second, we obtain the linear equation:
$$d+d'=2(a+b+c) .$$
So, if $a,b,c,d$ are integers, so is  $d'$.
Since the curvature of every circle from the second generation or later is $d'$ for some $4$ mutually tangent circles
of curvatures $a,b,c,d$ from the previous generation, we deduce:
\begin{Thm}[Soddy 1937]
If the initial 4 circles in an Apollonian packing $\mathcal P$ have integral curvatures, 
$\P$ is integral.
\end{Thm}

Combined with Descartes' theorem, for any integral solution of
$2(a^2+b^2+c^2+d^2)=(a+b+c+d)^2 $, there exists an integral Apollonian packing!
Because the smallest positive curvature must be at least $1$, an integral Apollonian packing cannot have arbitrarily large circles.
In fact, any integral Apollonian packing is either bounded or lies between two parallel lines.

\begin{figure}
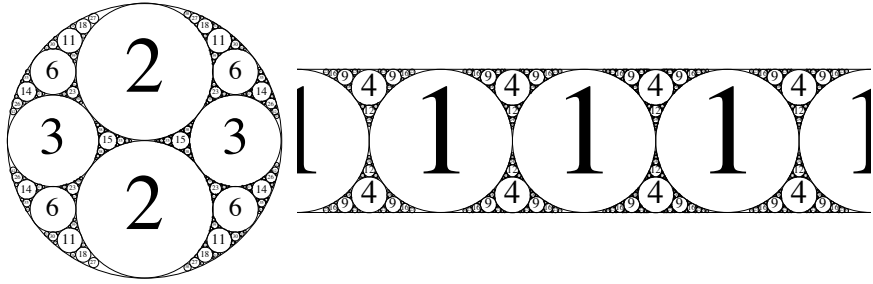

 \includegraphics [width=1.5in]{newer6}
 \includegraphics [width=3in]{InfPack}
 \caption{Integral Apollonian packings}
\end{figure}

For a given integral Apollonian packing $\P$,
it is natural to inquire about its Diophantine properties
 such as
\begin{Que}
\begin{itemize}
 \item Are there infinitely many circles with prime curvatures?
\item Which integers appear as curvatures? 
\end{itemize}
\end{Que}

 %\begin{itemize} \item
 %{\color{red!90!black} the infinitude of
%prime circles} (or pairs of tangent prime circles);
%\item the
%analogues of  {\color{red!90!black} prime number theorem}.
%\end{itemize}
%\bigskip

We call $\mathcal P$ 
primitive, if $\op{g.c.d}_{C\in \P}(\op{curv}(C))=1$.
We call  a circle is  {prime} if its curvature is a prime number, and a pair of tangent prime circles will be called
{twin prime} circles.
There are no triplet primes of three mutually tangent circles, all having
odd prime curvatures.

\begin{Thm}[Sarnak 07]\label{sarn}
 There are {infinitely many} prime circles as well as twin prime circles
  in any primitive integral Apollonian packing.
\end{Thm}
In the rest of this section, we let $\P$ be a bounded  primitive integral Apollonian packing.
Theorem \ref{sarn} can be viewed as an analogue of the infinitude of prime numbers. In order to formulate
what can be considered as an analogue of the prime number theorem,
we
set $$\Pi_T(\P):=\# \{ \text{prime }C\in \mathcal P:\op{curv}(C)\le  T\}$$
and 
$$\Pi_T^{(2)}(\P):=\# \{\text{twin primes } C_1, C_2 \in \P: \op{curv}(C_i)\le  T \} .$$

Using the sieve method based
on heuristics on the randomness of M\"obius function, Fuchs and Sanden \cite{FS} conjectured:
\begin{Con}[Fuchs-Sanden]\label{fsc}
$$\Pi_T(\P)\sim c_1 \frac{N_\P(T)}{\log T};\quad \Pi_T^{(2)}(\P)\sim c_2 \frac{N_\P(T)}{(\log T)^2}   $$ 
where $c_1>0$ and $c_2>0$ can be given explicitly.
\end{Con}
{}

%%%%% since n_T(p) is the count for all integer curv,
%%%%% the order for the counting ftn for primes is expected to be n_T(p) over log t and
%%%%%for twin primes,..expected to be n_T(p)/(log T)^2 but the conj gives a des of the constants
%%%%%%c_1 and c_2 

Based on the breakthrough of  {Bourgain, Gamburd, Sarnak} \cite{BGS} proving that
 the Cayley graphs of congruence quotients of the integral Apollonian group form an expander family,
together with  {Selberg}'s upper bound sieve, we obtain upper bounds of true order
of magnitude:

\begin{Thm}[Kontorovich-O. \cite{KO}]\label{primek}
For $T\gg 1$,
\begin{itemize}
\item $ \Pi_T(\P) \ll \frac{T^\alpha}{\log T}$;
\item $\Pi_T^{(2)}(\P) \ll \frac{T^\alpha}{(\log T)^2}$ .
\end{itemize}
\end{Thm}

The lower bounds for Conjecture \ref{fsc} are still open and very challenging. However a problem which is more amenable to current technology
 is to count curvatures without multiplicity.
Our counting Theorem \ref{kot} for circles says that
the number of integers at most $T$ arising as
curvatures of circles in integral $\P$ counted with multiplicity,
is of order $T^{1.3...}$. 
 So one may hope that a positive density (=proportion) of integers arises as curvatures,
as conjectured by Graham, Lagarias, Mallows, Wilkes, Yan (Positive density conjecture) \cite{G2}.

\begin{Thm}[Bourgain-Fuchs \cite{BF}]
For a primitive integral Apollonian packing $\P$,
$$\#\{\op{curv} (C)\le  T: C\in \P\} \gg T.$$
\end{Thm}

A stronger conjecture, called the Strong Density conjecture, of Graham et al. says that
every integer occurs as the value of a curvature of a circle in $\P$,
unless there are congruence obstructions.
Fuchs \cite{Fu1} showed that the only congruence obstructions are modulo $24$, and hence
the strong positive density conjecture (or the local-global principle conjecture) says that every sufficiently large integer which is congruent to a curvature of a circle in $\P$ modulo $24$ must occur
as the value of a curvature of some circle in $\P$.
This conjecture is still open, but there is now a stronger version of the positive density theorem:
\begin{Thm}[Bourgain-Kontorovich \cite{BK}] For a primitive integral Apollonian packing $\P$,
$${\#\{\op{curv} (C)\le  T: C\in \P\}} \sim \frac{\kappa(\P)}{24}\cdot T $$
where $\kappa(\P)>0$ is the number of residue classes mod $24$ of curvatures of $\P$. \end{Thm}

Improving Sarnak's result on the infinitude of prime circles, Bourgain showed that
a positive fraction of prime numbers appear as curvatures in $\P$.
\begin{Thm}[Bourgain  \cite{Bo}]
 $$\#\{\op{prime} \op{curv} (C) \le  T: C\in \P\} \gg  \frac{T}{\log T} .$$
\end{Thm}

\noindent{\bf Integral Apollonian group.}
In studying the Diophantine properties of integral Apollonian packings, we
work with the integral Apollonian group, rather than the geometric Apollonian group which was defined in section \ref{whitehead}.

We call a quadruple $(a,b,c,d)$ a Descartes quadruple if it represents curvatures of four mutually tangent circles
(oriented so that their interiors are disjoint) in the plane.
By Descartes' theorem, any Descartes quadruple $(a,b,c,d)$ lies on the cone
$Q(x)=0$,
where
 $Q$ denotes the so-called Descartes
quadratic form
$$Q(a,b,c,d)=2(a^2+b^2+c^2+d^2)-{(a+b+c+d)}^2.
$$

The quadratic form $Q$ has signature $(3,1)$ and hence over the reals, the orthogonal group $\O_Q$ is isomorphic to
$\O(3,1)$, which is the isometry group of the hyperbolic $3$-space $\bH^3$.

We observe that if $(a,b,c,d)$ and $(a,b,c,d')$ are
Descartes quadruples, then $d'=-d+2(a+b+c)$ and hence $(a,b,c,d')=(a,b,c,d)S_4$
where
$$S_1=\begin{pmatrix}  -1&0&0&0 \\2&1&0&0\\
2&0&1&0\\ 2&0&0&1 \end{pmatrix},\quad
 S_2=\begin{pmatrix} 1&2&0&0\\
 0&-1& 0&0 \\
0&2&1&0\\ 0&2&0&1 \end{pmatrix}, $$
$$
 S_3=\begin{pmatrix} 1&0&2&0\\
 0&1&2&0\\ 0&0& -1 &0 \\
 0&0&2&1 \end{pmatrix}, \quad  S_4=\begin{pmatrix} 1&0&0&2\\
 0&1&0&2 \\ 0&0&1&2\\ 0&0&0&-1 \end{pmatrix}. $$
Now the integral Apollonian group $\mathcal A$ is generated by those four reflections $S_1,S_2,S_3,S_4$ in $\operatorname{GL}_4(\z)$ and  
one can check that $\mathcal A<\O_Q(\z)$. 

Fixing an integral Apollonian circle packing $\P$,
all Descartes quadruples associated to $\P$ is a single $\A$-orbit in the cone $Q=0$. Moreover if we choose a root quadruple
$v_{\P}$ from $\P$, which consists of curvatures of four largest mutually tangent circles,
any reduced word $w_n=v_{\P}S_{i_1}\cdots S_{i_n}$ with $S_{i_j}\in \{S_1,S_2,S_3,S_4\}$
is obtained from $w_{n-1}=v_{\P}S_{i_1}\cdots S_{i_{n-1}}$ by changing precisely one entry and this new entry is
the maximum entry of $w_n$, which is the curvature of a precisely one new circle added at the $n$-th generation \cite{G2}.
This gives us the translation of the circle counting problem for a {\it bounded} Apollonian packings as the orbital counting problem of an $\A$-orbit in a cone $Q=0$:
$$N_{T}(\P)=\#\{v\in v_{\P} \A: \|v\|_{\op{max}} \le T\}+3.$$

The integral Apollonian group $\A$ is isomorphic to the geometric Apollonian group $\A_\P$ (the subgroup generated by four inversions with respect to the dual circles of four mutually tangent circles in $\P$):
there exists an explicit isomorphism between the orthogonal group $\O_Q$ and $\text{M\"ob}(\hat{\mathbb C})$
which maps the integral Apollonian group $\A$ to the geometric Apollonian group $\A_\P$.
In particular, $\A$ is a subgroup $\O_Q(\z)$ which is of infinite index and Zariski dense in $\O_Q$. Such a subgroup
is called a thin group. Diophantine properties of an integral Apollonian packing is now reduced to the study of
Diophantine properties of an orbit of the thin group $\A$. Unlike orbits under an arithmetic subgroup (subgroups of $\O_Q(\z)$ of finite index)
which has a rich theory of automorphic forms and ergodic theory, the study of thin groups has begun very recently, but with a great success.
In particular, the recent developments in expanders  is one of key ingredients
in studying primes or almost primes in thin orbits (see \cite{Bo}).

\section{Expanders and Sieve}
All graphs will be assumed to be simple (no multiple edges and no loops) and connected in this section.
For a finite $k$-regular graph $X=X(V,E)$ with $V=\{v_1, \cdots, v_n\}$ the set of vertices and $E$ the set of edges, the adjacency matrix $A=(a_{ij})$ is defined by $a_{ij}=1$ if $\{v_i, v_j\}\in E$ and $a_{ij}=0$ otherwise.
Since $A$ is a symmetric real matrix, it has $n$ real eigenvalues: $\lambda_0(X) \ge \lambda_1(X)\ge \cdots \ge  \lambda_{n-1}(X)$.
As $X$ is simple and connected,  the largest eigenvalue $\lambda_0(X)$ is given by $k$ and has multiplicity one.

\begin{Def} A family of $k$-regular graphs $\{X_i\}$ with ( $\# X_i \to \infty$) is called an {\it expander} family if there exists an $\e_0>0$ such that
$$\sup_i \lambda_1(X_i) \le k-\e_0.$$\end{Def} 
Equivalently, $\{X_i\}$ is an expander family if there exists a uniform positive lower bound
for the Cheeger constant (or isoperimetric constant) 
$$h(X_i):=\min_{0<\# W\le \# X_i /2} \frac{\# \partial (W)}{\# W}$$ where 
$\partial (W)$ means the set of edges with exactly one vertex in $W$.
Note that the bigger the Cheeger constant is, the harder it is to break the graph into two pieces.
Intuitively speaking, an expander family is a family of sparse graphs (as the regularity $k$ is fixed) with high connectivity properties (uniform lower bound
for the Cheeger constants).

Although it was known that there has to be many expander families using probabilistic arguments due to Pinsker,
the first explicit construction of an expander family is due to Margulis in 1973 \cite{MaE} using the representation theory of a simple algebraic group and automorphic form theory.
We explain his construction below; strictly speaking, what we describe below is not exactly same as his original construction but
the idea of using the representation theory of an ambient algebraic group is the main point of his construction as well as in
the examples below.

Let $G$ be a connected simple non-compact real algebraic group defined over $\q$, with a fixed $\q$-embedding into $\SL_N$.
Let $G(\z):=G\cap \SL_N(\z)$ and $\G<G(\z)$ be a finitely generated subgroup. For each positive integer $q$,
the principal congruence subgroup $\G(q)$ of level $q$ is defined to be $\{\gamma\in \G: \gamma =e \;\mod q\}$.

Fix a finite symmetric generating subset $S$ for $\G$. Then $S$ generates the group $\G(q)\ba \G$ via the canonical projection.
We denote by $X_q:=\mathcal C(\G(q)\ba \G, S)$ the Cayley graph of the group $\G(q)\ba \G$ with respect to $S$, that is, vertices of $X_q$
are elements of $\G(q)\ba \G$ and two elements $g_1, g_2$ form an edge if $g_1=g_2 s$ for some $s\in S$.
Then $X_q$ is a connected $k$-regular graph for $k=\# S$.
Now a key observation due to Margulis is that if $\G$ is {\it of finite index} in $G(\z)$, or equivalently if $\G$ is a lattice in $G$,
then the following two properties are equivalent:
 for any $I\subset \mathbb N$,
\begin{enumerate}
\item The family $\{X_q:q\in I\}$ is an expander;
\item  The trivial representation $1_G$ is isolated 
in the sum $\oplus_{q\in I} L^2(\G(q)\ba G)$ in the Fell topology of the set of unitary representations of $G$.\end{enumerate}
 We won't give a precise definition of the Fell topology, but just say that the second property is equivalent to the following:
for a fixed compact generating subset $Q$ of $G$, there exists $\e>0$ (independent of $q\in I$)
 such that any unit vector $f\in L^2(\G(q)\ba G)$ satisfying
$\max_{q\in Q} \|q.f-f\| <\e$ is $G$-invariant, i.e., a constant.
Briefly speaking, it follows almost immediately from the definition of an expander family that the family $\{X_q\}$ is an expander if and only if the trivial representation $1_{\G}$ of $\G$ is isolated in 
the sum $\oplus_q L^2(\G(q)\ba \G)$. On the other hand,the induced representation of $1_\G$ from $\G$ to $G$ is $L^2(\G\ba G)$,
which contains the trivial representation $1_G$, if $\G$ is a lattice in $G$. Therefore, by the continuity of the induction process, the weak containment of
$1_\G$ in $\oplus_q L^2(\G(q)\ba \G)$ implies the weak-containment of $1_G$ in 
$\oplus_q L^2(\G(q)\ba G)$, which explains why (2) implies (1).

The isolation property of $1_G$ as in (2) holds for $G$; if the real rank of $G$ is at least $2$ or $G$ is a rank one group
of type $Sp(m,1)$ or $F_4^{-20}$, $G$ has the so-called Kazhdan's property (T) \cite{Ka}, which says that the trivial representation of $G$ is isolated in the 
whole unitary dual of $G$. When $G$ is isomorphic to $\SO(m,1)$ or $\SU(m,1)$ which do not have Kazhdan's property (T),
the isolation of the trivial representation is still true in the subset of all automorphic representations $L^2(\G(q)\ba G)$'s, due to the work
of Selberg, Burger-Sarnak \cite{BS} and Clozel \cite{Cl}. This latter property is referred as the phenomenon that
$G$ has property $\tau$ with respect to the congruence family $\{\G(q)\}$.

Therefore, we have:
\begin{Thm} \label{expo} If $\G$ is of finite index in $G(\z)$, then
the family $\{X_q=\mathcal C (\G(q)\ba \G, S) :q \in \N \}$ is an expander family.
\end{Thm}

In the case when $\G$ is of infinite index, the trivial representation is not contained in $L^2(\G(q)\ba G)$, as the constant function is
not square-integrable, and the above correspondence cannot be used, and deciding
whether $X_q$ forms an expander or not for a thin group was a longstanding open problem. For instance, if
$S_k$ consists of four matrices $\begin{pmatrix}
1 &\pm k\\ 0 & 1\end{pmatrix}$ and $\begin{pmatrix}
1 &0 \\ \pm k  & 1\end{pmatrix}$,  then the group $\G_k$ generated by $S_k$
has finite index only for $k=1,2$  and hence
we know the family $\{X_q(k) = \mathcal C (\G_k(q)\ba \G_k, S_k)\}$ forms an expander for $k=1,2$ by Theorem \ref{expo} but the subgroup $\G_3$  generated by
$S_3$ has infinite index in $\SL_2(\z)$ and  it was not known whether
$\{X_q(3)\}$ is an expander family until the work of Bourgain, Gamburd and Sarnak \cite{BGS}.

\begin{Thm}[Bourgain-Gamburd-Sarnak \cite{BGS}, Salehi-Golsefidy-Varju \cite{SV}] \label{exp} Let $\G<G(\z)$ be a thin subgroup, i.e., $\G$ is Zariski dense in $G$.
Let $S$ be a finite symmetric generating subset of $\G$. Then
$\{\mathcal C(\G(q)\ba \G, S): \text{q: square-free}\}$ forms an expander family.
\end{Thm}
If $G$ is simply connected in addition, the strong approximation theorem of Matthews, Vaserstein and Weisfeiler \cite{MVW}
says that there is a finite set $\mathcal B$ of primes such that for all $q$ with no prime factors from $\mathcal B$,
$\G(q)\ba \G$ is isomorphic to the finite group $G(\mathbb Z/ q\mathbb Z)$ via the canonical projection $\G \to G(\z/q\z)$; hence the corresponding
 Cayley graph $\mathcal C (G(\mathbb Z/ q\mathbb Z), S)$ is connected.
Similarly, Theorem \ref{exp} says that the Cayley graphs $\mathcal C (G(\mathbb Z/ q\mathbb Z), S)$ with $q$ square-free and with no factors from $\mathcal B$
are highly connected, forming an expander family; called the {\it super-strong approximation Theorem}.

The proof of Theorem \ref{exp} is based on additive combinatorics and Helfgott's work on approximate subgroups \cite{He} and generalizations made by Pyber-Szabo \cite{PS} and
Breuillard-Green-Tao \cite{BGT} (see also \cite{Gr}).

The study of expanders has many surprising applications in various areas of mathematics (see \cite{Sa1}). We describe its application in sieves, i.e., in the study of primes. For motivation, we begin by considering an integral polynomial $f\in \z[x]$.
The following is a basic question:
\begin{center} {\it Are there infinitely many integers \text{$n\in \z$} such that \text{$f(n)$} is prime}?\end{center}

\begin{itemize}
\item If $f(x)=x$, the answer is yes; this is the infinitude of primes. 
\item
If $f(x)=ax+b$, the answer is yes if and only if $a,b$ are co-prime. This is Dirichlet's theorem.
\item If $f(x)=x(x+2)$, then there are no primes in $f(\z)$ for an obvious reason. On the other hand, Twin prime conjecture says that there are infinitely many $n$'s such that $f(n)$ is a product of at most $2$ primes. Indeed, Brun introduced what is called Brun's combinatorial sieve
to attack this type of question, and proved that 
there are infinitely many $n$'s such that $f(n)=n(n+2)$ is $20$-almost prime, i.e., a product of at most $20$ primes. His approach was improved by Chen \cite{Ch}
who was able to show such a tantalizing theorem that there are 
infinitely many $n$'s such that $f(n)=n(n+2)$ is $3$-almost prime.
\end{itemize}

In view of the last example, the correct question is formulated as follows: 
\begin{center} {\it Is there $R<\infty $ such that the set of \text{$n\in \z$} such that
 \text{$f(n)$} is R-almost prime is infinite?}\end{center}

Bourgain, Gamburd and Sarnak \cite{BGS2} made a beautiful observation that Brun's combinatorial sieve can also  be implemented
for orbits of $\G$ on an affine space via affine linear transformations
   and the expander property of the Cayley graphs of the congruence quotients of $\G$ provides a crucial input needed in executing the sieve machine.

Continuing our setup  that
$G\subset \SL_N$ and $\G <G(\z)$, we consider the orbit $\mathcal O=v_0\G\subset \z^N$ for a non-zero $v_0\in \z^N$ and let $f\in \q[x_1, \cdots, x_N]$ such that $f(\mathcal O)\subset \z$.
\begin{Thm}[Bourgain-Gamburd-Sarnak \cite{BGS2}, Sarnak-Salehi-Golsefidy \cite{SS}] There exists $R=R(\mathcal O, f)\ge 1$ such that
the set of vectors $v\in \mathcal O$ such that $f(v)$ is $R$-almost prime is Zariski dense in $v_0G$.
\end{Thm}

%Although this is a very strong result, Zariski-density is a weak notion.
We ask the following finer question:
\begin{center}\it{ Describe the distribution of the set \text{$\{v\in \mathcal O: f(v) \text{ is $R$-almost prime}\}$}
in the variety $v_0G$. }\end{center}
In other words, is the set in concern focused in certain directions in $\mathcal O $ or {\it equi-distributed} in $\mathcal O$?
This is a very challenging question at least in the same generality as the above theorem, but when $G$
is the orthogonal group of the Descartes quadratic form $Q$, $Q(v_0)=0$, and $\G$ is the integral Apollonian group, 
we are able to give more or less a satisfactory answer by \cite{KO} and \cite{LOA}. More generally, we have the following:
Let $F$ be an integral quadratic form of signature $(n,1)$ and let $\G<\SO_F(\z)$ be a geometrically finite Zariski dense
subgroup. Suppose that the critical exponent $\delta$ of $\G$ is bigger than $(n-1)/2$ if $n=2,3$ and bigger than $n-2$ if $n\ge 4$.
Let $v_0\in \z^{n+1}$ be non-zero and $\mathcal O:=v_0\G$. We also assume that the skinning measure
associated to $v_0$ and $\G$ is finite. 
\begin{Thm} [Mohammadi-O. \cite{MO2}]\label{efc} Let $f=f_1\cdots f_k\in \q[x_1, \cdots, x_{n+1}]$ be a polynomial with each
$f_i$ absolutely irreducible and distinct with rational coefficients and $f_i(\mathcal O)\subset \z$. Then we construct an explicit locally finite measure $\mathcal M$
on the variety $v_0G$, depending on $\G$
such that for any family $\mathcal B_T$ of subsets in $v_0G$  which is {\it effectively well-rounded} with respect to $\mathcal M$, we have
\begin{description}
\item[(1) Upper bound] $\#\{v\in \mathcal O\cap \mathcal B_T: \text{each $f_i(v)$ is prime}\}\ll \frac{\mathcal M(\B_T)}{(\log \mathcal M(\B_T))^k};$
\item[(2) Lower bound] Assuming further that $\max_{x\in \mathcal B_T} \|x\| \ll \mathcal M(B_T)^\beta$
for some $\beta>0$,
there exists $R=R(\mathcal O, f)>1$ such that
$$\#\{v\in \mathcal O\cap \mathcal B_T:\text{$f(v)$ is $R$-almost prime} \}\gg  
\frac{\mathcal M(\B_T)}{(\log \mathcal M(\B_T))^k}.$$
\end{description}
\end{Thm}

The terminology of $\B_T$ being effectively well-rounded with respect to $\mathcal M$ means that there exists $p>0$ such that
for all small $\e>0$ and for all $T\gg 1$,
the  $\mathcal M$-measure
 of the $\e$-neighborhood of  the boundary of $\B_T$ is at most of order $O(\e^p\mathcal M(\B_T))$ with the implied constant independent of $\e$ and $T$.
  For instance, the norm balls $\{v\in v_0G: \|v\|\le T\}$ and many sectors are effectively well-rounded (cf. \cite{MO2}).

When $\G$ is of finite index, $\mathcal M$ is just a $G$-invariant measure on $v_0G$ and
 this theorem was proved earlier by Nevo-Sarnak \cite{NS} and  Gorodnik-Nevo \cite{GN}.

If $Q$ is the Descartes quadratic form, $\mathcal A$ is the integral Apollonian group, and
$\mathcal B_T=\{v\in \br^{4}: Q(v)=0, \|v\|_{\max}\le T\}$ is the max-norm ball,
then for any primitive integral Apollonian packing $\P$,
the number of prime circles in $\P$ of curvature at most $T$ is bounded  by
$$\sum_{i=1}^4 \#\{v\in v_{\P}\mathcal A\cap \B_T, f(v):=v_i \text{ prime}\}$$
which is  bounded by $\frac{\mathcal M(\B_T)}{\log (\mathcal M(\B_T))}$ by Theorem \ref{efc}. 
Since we have $\mathcal M(\B_T) =c \cdot T^\alpha +O(T^{\alpha-\eta})$ where $\alpha=1.305...$ is the critical exponent of $\mathcal A$,
this gives an upper bound $T^\alpha/\log T$ for the number of prime circles of curvature at most $T$, as stated in Theorem \ref{primek}.
The upper bound for twin prime circle count can be done similarly with $f(v)=v_{i}v_j$.

\newcommand{\cA}{\mathsf A}\newcommand{\X}{\mathcal X}\renewcommand{\deg}{\op{deg}}
\newcommand{\M}{\mathcal M}
\renewcommand{\O}{\mathcal O}

Here are a few words on
Brun's combinatorial sieve and its use  in Theorem \ref{efc}.
 Let $\cA=\{a_m\}$ be a sequence of non-negative numbers and let $B$ be a finite set of primes.
For $z >1$,
let $P_z=\prod_{p\notin B, p<z}p$ and $S(\cA,P_z):=\sum_{(m,P_z)=1} a_m .$
To estimate $S_z:=S(\cA,P_z)$, we need to understand how $\cA$ is distributed along arithmetic progressions.
For $q$ square-free, define
$$
\cA_q:=\{a_m\in\cA:n\equiv0(q)\}
$$
and set
$
|\cA_q|:=\sum_{m\equiv0(q)}a_m.
$

We use the following combinatorial sieve  (see \cite[Theorem 7.4]{HR}):
\begin{thm}\label{sie} \begin{itemize}
\item[$(A_1)$] For $q$ square-free with no factors in $B$,
suppose that $$|\cA_q|=g(q) \X + r_q(\cA)$$
where $g$ is a function on square-free integers with $0\le g(p)<1$,
 $g$ is multiplicative outside $B$, i.e.,
$g(d_1d_2)=g(d_1)g(d_2)$ if $d_1$ and $d_2$ are square-free integers  with $(d_1, d_2)=1$ and $(d_1 d_2, B)=1$,
and for some $c_1>0$, $g(p)<1-1/c_1$ for all prime $p\notin B$.
\item[$(A_2)$]  $\cA$ has level distribution $D$, in the sense that for some $\e>0$ and $C_\e>0$,
$$\sum_{q<D}|r_q(\cA)|\le  C_\e \X^{1-\e}.$$
\item[$(A_3)$]  $\cA$ has sieve dimension $k$ in the sense that there exists $c_2>0$ such that for all $2\le w\le z$,
$$-c_2\le \sum_{(p,B)=1, w\le p\le z} g(p)\log p -r \log \frac{z}{w} \le c_2 .$$
Then for $s>9r$, $z=D^{1/s}$ and $\X$ large enough,
$$S(\cA,P_z)\asymp \frac{\X}{(\log \X)^k} .$$
           \end{itemize}
\end{thm}

For our orbit $\O=v_0\G$ and $f$ as in Theorem \ref{efc},  we set
$$ a_m(T):=\#\{x\in  \mathcal O\cap \B_T: f(x)=m\};$$
$$ \G_{v_0}(q):=\{\gamma\in \G: v_0\gamma \equiv v_0 \; (q)\},$$
$$|\cA(T)|:=\sum_m a_m(T)=\# \mathcal O\cap \B_T;$$
$$ |\cA_q(T)|:=\sum_{m\equiv 0 (q)} a_m(T)= \#\{x\in  \mathcal O\cap \B_T: f(x)\equiv 0 \;(q)\}.$$

Suppose we can verify the sieve axioms for these sequences $\cA_q(T)$ and $z$ of order $T^\eta$. 
Observe that if $(f(v)=m, P_z)=1$, then all prime factors of $m$
have to be at least of order $z=T^\eta$. It follows that if $f(v)=m$ has $R$ prime factors, then
$T^{\eta R}\ll m\ll T^{\text{degree}(f)}$, and hence $R\ll (\text{degree f})/\eta$.
Therefore,  
$S_z:=\sum_{(m, P_z)=1}a_m(T)$ gives an estimate of the number of
 all $v\in \mathcal O$ such that $f(v)$ is $R$-almost prime for $R= \text{(degree f)}/\eta$.

In order to verify these sieve axioms for $\mathcal O=v_0\G$, we replace $\G$ by its preimage under the spin cover $\tilde G$ of $G$, so that
$\G$ satisfies the strong approximation property that $\G(q)\ba \G=\tilde G(\z/q\z)$ outside a fixed finite set of primes.
The most crucial condition
is to understand the distribution of $a_m(T)$'s along the arithmetic progressions, i.e., $\sum_{m=0 \;( q)} a_m(T)$ for all square-free integers $q$, more precisely,
we need to have a uniform control on the remainder term
$r_q$ of $\cA_q(T)=\sum_{m=0 \;(q)} a_m= g(q)\X +r_q$ such as $r_q \ll \X^{1-\e}$ for some $\e >0$ independent of $q$.
By writing $$\cA_q(T)=\sum_{\gamma\in \G_{v_0}(q)\ba \G, f(v_0\gamma) =0 \; (q)} \# (v_0 \G_{v_0}(q)\gamma \cap \B_T)$$
 the following uniform counting estimates provide such a control on the remainder term:

\begin{Thm}  [Mohammadi-O. \cite{MO2}] \label{lastt} Let $\Gamma$ and $\B_T$ be as in Theorem \ref{efc}.
For any $\gamma\in \G$ and any square-free integer $q$,
$$\# v_0 \G(q) \gamma \cap \mathcal B_T = \frac{c_0}{[\G: \G(q)]} \mathcal M (\mathcal B_T) + O(\mathcal M (\mathcal B_T)^{1-\e})$$
where $c_0>0$ and $\e>0$ are independent over all $\gamma\in \G$ and $q$.
\end{Thm}

A basic ingredient of Theorem \ref{lastt} is a uniform spectral gap for the Laplacian acting on $L^2(\G(q)\ba \bH^n)$. Note that
zero is no more the base-eigenvalue of the Laplacian when $\G(q)$ is a thin group, but $\delta (n-1-\delta)$ is by 
Sullivan \cite{Sullivan1984} and
Lax-Phillips \cite{LP}.
 However, the expander result (Theorem \ref{exp})
implies a uniform lower bound for the gap between the base eigenvalue $\delta(n-1-\delta)$ and the next one; this transfer property was
obtained by Bourgain, Gamburd and Sarnak.
As explained in section \ref{bmss},
the mixing of frame flow of the Bowen-Margulis-Sullivan measure is a crucial ingredient in obtaining the main term in Theorem \ref{lastt},
and the (uniform) error term in the counting statement of Theorem \ref{lastt}
is again a consequence of a uniform error term in the effective mixing of frame flow, at least under our hypothesis on $\delta$.

\end{document}